\documentclass[11pt,a4paper]{article}
\usepackage[utf8]{inputenc}
\usepackage{amssymb,amsmath}
\usepackage{graphicx}
\usepackage{wrapfig}
\usepackage{amsfonts}
\usepackage{fullpage}
\usepackage{caption}
\usepackage{subcaption}
\usepackage{hyperref}
\usepackage{amsthm}
\numberwithin{equation}{section}
\usepackage[capitalise]{cleveref}
\usepackage{unicode}
\usepackage{fancyhdr}
\usepackage{algorithm}
\usepackage{algorithmic}

\title{Investigation of optimal control problems governed by a time-dependent Kohn-Sham model}
\date{\today}

\hypersetup{colorlinks=true,
linkcolor=black,
citecolor=black,
pdftitle={Optimal control with TDKS},
pdfauthor={Martin Sprengel},
pdfsubject={},
pdfkeywords={}
}

\newcommand{\R}{\ensuremath{\mathbb{R}}}
\newcommand{\N}{\ensuremath{\mathbb{N}}}
\newcommand{\C}{\ensuremath{\mathbb{C}}}

\newcommand{\skalar}[2]{\left( #1,\,#2 \right)}
\newcommand{\scalar}[3][L^2(\Omega;\C^N)]{\left(#2,\,#3\right)_{#1}}
\newcommand{\scalarT}[2]{\left( #1,\,#2 \right)_{L^2(0,T;\R)}}
\newcommand{\scalarTH}[2]{\left( #1,\,#2 \right)_{H^1(0,T;\R)}}

\newcommand{\scalarC}[2]{\left(#1,\,#2\right)_{\C}}
\newcommand{\scalarCN}[3][L^2(\Omega;\C^N)]{\left(#2,\,#3\right)_{#1}}

\newcommand{\dd}{\mathrm{d}}

\newcommand{\abl}[2]{\frac{\dd #1}{\dd #2}}
\newcommand{\pa}{\partial}
\newcommand{\pabl}[2]{\frac{\pa #1}{\pa #2}}

\renewcommand{\Im}{\operatorname{Im}}
\renewcommand{\Re}{\operatorname{Re}}

\newcommand{\F}[1]{\ensuremath{\mathcal{F}\left[#1\right]}}

\newcommand{\D}{\ensuremath{\operatorname{D}}}

\newcommand{\co}[1]{\overline{#1}}

\newcommand{\Lop}{\mathcal{L}}
\newcommand{\xone}{x_1} 
\newcommand{\xtwo}{x_2}

\newtheorem{theorem}{Theorem}
\newtheorem{lemma}[theorem]{Lemma}

\newtheorem{definition}[theorem]{Definition}
\theoremstyle{definition}
\newtheorem*{remark}{Remark}

\newtheorem*{assumption*}{Assumption}

\begin{document}
\author{
M.~Sprengel\thanks{Institut f\"ur Mathematik, Universit\"at W\"urzburg, Emil-Fischer-Strasse 30,
97074 W\"urzburg, Germany ({\tt martin.sprengel@mathematik.uni-wuerzburg.de}).}
         \and
G. Ciaramella\thanks{
Section de mathématiques, 
Université de Genève, 2-4 rue du Lièvre
1211 Genève 4, Switzerland ({\tt gabriele.ciaramella@unige.ch}).}
         \and 
A.~Borz{\`i}\thanks{Institut f\"ur Mathematik, Universit\"at W\"urzburg, Emil-Fischer-Strasse 30,
97074 W\"urzburg, Germany ({\tt alfio.borzi@mathematik.uni-wuerzburg.de}).}
       }

\maketitle

\begin{abstract}
Many application models in quantum physics and chemistry require to control multi-electron systems to achieve a desired target configuration. 
This challenging task appears possible in the framework of time-dependent density functional theory (TDDFT) that allows to describe these systems while avoiding the high dimensionality resulting from the multi-particle Schrödinger equation. For this purpose, the theory and numerical solution of optimal control problems governed by a Kohn-Sham TDDFT model are investigated, 
considering different objectives and a bilinear control mechanism.  
Existence of optimal control solutions and their characterization as solutions to Kohn-Sham TDDFT optimality systems are discussed. To validate this control framework, a time-splitting discretization of the optimality systems and a nonlinear conjugate gradient scheme are implemented. Results of numerical experiments demonstrate the computational capability of the proposed control approach. 
\end{abstract}

\section{Introduction}
Many models of interest in quantum physics and chemistry consist of multi-particle systems, that can be modelled by the multi-particle Schrödinger equation (SE). However, the space dimensionality of this equation increases linearly with the number of particles involved and the corresponding computational cost increases exponentially, thus making the use of the multi-particle SE prohibitive. This fact has motivated a great research effort 
towards an alternative formulation to the SE description that allows to compute the observables of a quantum multi-particle system using particle-density functions. This development, which starts with the Thomas-Fermi theory in 1927, reaches a decisive point with the works of Hohenberg, Kohn, and Sham \cite{HohenbergKohn:PhysRev:1964, KohnSham1965} that propose an appropriate way to replace a system of $N$ interacting particles in an external potential  $V_{ext}$ by another system of non-interacting particles with an external potential $V_{ext}+V_{Hxc}$, such that the two models have the same electronic density. These works and many following ones focused on the computation of stationary (ground) states and obtained successful results that motivated the extension of the density function theory (DFT) theory to include time-dependent phenomena. This extension was first proposed by Runge and Gross in \cite{RungeGross1984}, and 
further investigated from a mathematical point of view in the work of van Leeuwen \cite{vanLeeuwen99}. We refer to \cite{EngelDreizler2011} 
for a modern introduction to DFT and to \cite{LectureNotesTDDFT} for 
an introduction to time-dependent DFT (TDDFT).

Similar to the stationary case, the Runge–Gross theorem proves that, given an initial wavefunction configuration, there exists a one-to-one mapping between the potential in which the system evolves and the density of the system. 
Therefore, under appropriate assumptions, given a SE for a system of interacting particles in an external potential, there exists another SE model, unique up to a purely time-dependent function in the potential \cite{vanLeeuwen99}, of a non-interacting system with an augmented potential whose solution provides the same density as the solution to the original SE problem. 
We refer to this TDDFT model as the time-dependent Kohn-Sham (TDKS) equation. 
Notice that the external potential $V_{ext}$ modelling the interaction of the particles 
(in particular, electrons) with an external (electric) field enters without modification in both the multi-particle SE model and the TDKS model. 

This latter fact is important in the design of control strategies for multi-particle quantum systems because control functions usually enter in the SE model as external time-varying potentials. 
Therefore control mechanisms can be determined in the TDDFT framework that are valid for the original multi-particle SE system. Recently, various quantum mechanical optimal control problems governed by the SE have been studied in the literature, see for example \cite{VonWinckelBorzi2008}, \cite{VonWinckelBorziVolkwein2009} and \cite{LiteraturSalomon}. Moreover, quantum control problems governed by TDDFT models have already been 
investigated; see, e.g., \cite{CastroWerschnikGross2012} and have been implemented in TDDFT codes as the well-known Octopus \cite{octopuscite}. 
However, the available optimization schemes are mainly based on less competitive Krotov's method and consider only finite-dimensional parameterized controls. Furthermore, much less is known on the theory of the TDDFT 
optimal control framework and on the use and analysis of more efficient optimization schemes that allow to compute control functions belonging to a much larger function space.  

We remark that the functional analysis of optimization problems 
governed by the TDKS equations and the investigation of optimization schemes requires the mathematical foundation of the governing model. At the best of our knowledge, only few contributions addressing this issue are available; we refer to \cite{RuggenthalerPenzVanLeeuwen2015, Jerome2015, SprengelCiaramellaBorzi2016} for results concerning the existence and uniqueness of solutions to the TDKS equations.

This work contributes to the field of optimal control theory for multi-particle quantum systems presenting a theoretical analysis of optimal control problems governed by the TDKS equation. 
To validate the proposed framework, we implement an efficient approximation and optimization scheme for these problems. 

This paper is organized as follows. 
In Section \ref{sec:model}, we illustrate multi-particle SE models and the Kohn-Sham (KS) approach to TDDFT. Since these models are less known in the PDE optimal control community, and the literature 
on these problems is sparse, we make a special effort to provide a detailed presentation and to present results that 
are instrumental for the discussion that follows. 
In Section \ref{sec:oct}, we state a class of optimal control problems and discuss the related first-order optimality systems. 
The details of the derivation of the optimality system are elaborated in the \cref{sec:DerivationOfOptSystem}. 
The analysis of the optimal control problems is presented in Section \ref{sec:theory}. We show existence of optimal solutions to the control problems and prove necessary optimality conditions. 
Section \ref{sec:implementation} is dedicated to suitable 
approximation and numerical optimization schemes are discussed. We consider time-splitting schemes \cite{BaoJinMarkowich2002,StrangFaou2014} 
and discuss their accuracy properties. To solve the optimality systems, we implement a nonlinear conjugate gradient scheme.
In Section \ref{sec:experiments}, results of numerical experiments are presented that demonstrate the effectiveness of the proposed control framework. A section of conclusions completes this work.

\section{The TDKS model}\label{sec:model}

In the Schrödinger quantum mechanics framework, the state of a $N$ electrons system is described by a wave function $\Psi$, whose time evolution is governed by the following Schrödinger equation (SE)
\begin{equation}
  i\frac{\partial}{\partial t}\Psi(x_1,\dotsc x_N,t)=H\Psi(x_1,\dotsc x_N,t),\label{intro:multiSE}
\end{equation}
where $x_1,\dotsc,x_N\in\R^n$ are the position vectors of the $N$ particles. 
We use atomic units, i.e. $\hbar=4\pi\epsilon_0=1$, $m_e=\frac{1}{2}$.

The Hamiltonian $H$ consists of a kinetic term, the Coulomb interaction between the charged particles (electrons), $\sum_{i<j}\frac{1}{|x_i-x_j|}$, and an external potential, $V_{ext}$. We have
\begin{align}\label{Model:TheModel}
	H=\sum_{i=1}^N \left(-\nabla_i^2+V_{ext}(x_i, t)\right)+\sum_{i<j}\frac{1}{|x_i-x_j|},
\end{align}
where $|\cdot|$ is the Euclidean norm and $\nabla_i$ is the $n$-dimensional vector gradient with respect to $x_i$.

The Pauli principle states that the wave function of a system of electrons has to be antisymmetric with respect to the exchange of two coordinates. For this purpose, given $N$ orthogonal single particle wave functions (orbitals), $\psi_j(x,t)$, that correspond to the Hamiltonian of the $j$th particle, $H_j=\left(-\nabla_j^2+V_{ext}(x_j, t)\right)$ one can build the following antisymmetric wave function 
\begin{align*}
	\Psi(x_1,x_2,\dotsc, x_N,t)&=\frac{1}{\sqrt{N!}}\det
	\begin{pmatrix}
		 \psi_1(x_1, t) &  \psi_2(x_1, t) & \cdots &  \psi_N(x_1,t)\\
		\vdots		& \vdots	& \vdots & \vdots\\
		 \psi_1(x_N, t) &  \psi_2(x_N, t) & \cdots &  \psi_N(x_N,t)
	\end{pmatrix}.
\end{align*}
This is called Slater determinant. This wave function solves (\ref{intro:multiSE}) if the particles do not interact. However, in the presence of an interaction potential, the solution to (\ref{intro:multiSE}) will be given by an infinite sum of Slater determinants.

This fact shows that the effort of solving a multi-particle SE increases exponentially with the number of particles. To avoid this curse of dimensionality, the approach of DFT is to consider, instead of the wave function on a $nN$-dimensional space, the corresponding electronic density defined of the physical space of $n$-dimensions 
given by  
\begin{align}
\rho(x,t) = N \int |\Psi(x, \dotsc x_N,t)|^2 \dd x_2 \cdots \dd x_N.
\end{align}
	For a Slater determinant, one finds that the corresponding density is as follows
	\begin{align}
	\rho(x,t)=\sum_{i=1}^N | \psi_i(x,t)|^2,
\end{align}
where $\psi_i$ represents the wave function of the $i$th particle.

The DFT approach of Kohn and Sham \cite{KohnSham1965} to 
model multi-particle problems was to replace the system of $N$ interacting particles, subject to an external potential $V_{ext}$, by another system of non-interacting particles subject to an 
augmented potential $V_{ext}+V_{Hxc}$, 
such that the two models provide the same density.
Van Leeuwen \cite{vanLeeuwen99} proved that such a system exists under appropriate conditions on the potentials and on the resulting densities. In particular, for this proof it is 
required that the wave function is twice continuously differentiable in space and analytic in time and the potential has to have finite expectation values and be differentiable in space and analytic in time.  

Based on this development, we consider the time-dependent Kohn-Sham system given by 
\begin{eqnarray}
i\pabl{ \psi_j}{t}(x,t)&=&\left(-\nabla^2+V_{ext}(x,t)+V_{Hxc}(x,\rho) \right)  \psi_j(x,t),\label{KSequations}\\
{ \psi_j}(x,t)&=&  \psi_j^0 (x), \qquad j=1, \dotsc,N. \label{KSinitcond}
\end{eqnarray}
It is not trivial to get an initial condition that appropriately represents the interacting system because it may not have a Slater determinant as starting wavefunction.
A common choice is to solve the ground state DFT problem and take the eigenstates corresponding to the lowest $N$ eigenvalues as $\psi_j^0$.

Notice that an TDKS system is formulated in $n$ spatial dimensions and consists of $N$ coupled Schrödinger equations.   
With the appropriate choice of $V_{Hxc}$, which contains the coupling through the dependence on $\rho$, the solution to \eqref{KSequations}--\eqref{KSinitcond} provides the correct density of the original system, so that all observables, which can be formulated in terms of the density, can be determined by this method.

The main challenge of the DFT framework is to construct KS potentials that encapsulate all the multi-body physics. One class of approximations is called the local density approximation (LDA) 
\cite{EngelDreizler2011}, because in this approach, the KS potential at some point $x$ only depends on the value of the density $\rho(x)$ at this specific point. We use the adiabatic LDA, which means that LDA is applied at every time separately such that $V_{Hxc}(x,t)=V_{Hxc}(\rho(x,t))$. 
Notice that, if one allows $V_{Hxc}(t)$ to depend on the whole history $V_{Hxc}(\tau)$, $0\leq \tau \leq t$, the resulting adjoint equation would be an integro-differential equation, which would be much more involved to solve.

We remark that the Kohn-Sham potential $V_{Hxc}$ is usually split into three terms, that is, the Hartree potential, and the exchange and correlation potentials as follows
\begin{align}
\label{SumPotentials}
	V_{Hxc}(x,\rho(x,t)) =V_H(x,\rho(x,t))+V_x(\rho(x,t))+V_c(\rho(x,t)).
\end{align}
The Hartree potential $V_H$ models electrons interaction due to the Coulomb force. This term dominates $V_{Hxc}$, while the other two parts represent quantum mechanical corrections. Later, 
we refer to the exchange-correlation part of the potential also 
as $V_{xc}=V_x+V_c$. Notice that the classical electric field is given by $\phi(x)=\int_\Omega \frac{\rho_c(y)}{|x-y|}\dd y$, where $x$ and $y$ are positions in space and $\rho_c$ is the charge density. Since we are using natural units, charge and particle densities become the same and we have the following 
\begin{align}
V_H(x,\rho)=\int_\Omega \frac{\rho(y)}{|x-y|}\dd y .
\end{align}

In quantum mechanics, the Pauli exclusion principle states that two electrons cannot share the same quantum state. This results in a repulsive force between the electrons. In DFT this feature is modeled by introducing two terms: the exchange and the correlation potentials. The exchange potential $V_x$ contains the Pauli principle for a homogeneous electron gas. In the LDA framework, it can be calculated explicitly as follows; see, e.g., \cite{ExchangePotential_Constantin_PRB, ParrYang1989}
\begin{align*}
	V_x^{2D}(\rho)=-\sqrt{\frac{8}{\pi}} \sqrt{\rho} \, , && V_x^{3D}(\rho)=-\sqrt[3]{\frac{3}{\pi}} \sqrt[3]{\rho}.
\end{align*}

However, quantum mechanics is not applicable for very short distances such that relativistic effects start to play a role. Therefore, we introduce a cut-off of the potential at unphysically large densities, while preserving all required properties in the range of validity of the DFT framework.

We define the exchange potential as follows
\begin{align}\label{Vxcuttoff}
	V_x:[0,\infty)\rightarrow [0,p(2R)], && V_x(\rho)=\alpha_n\begin{cases}
		\sqrt[n]{\rho}	& \rho \leq R\\
		p(\rho) & R<\rho<2R\\ 
		p(2R)	& \rho \geq 2 R
	\end{cases},
\end{align}
where 
\begin{equation}
p(\rho)=\frac{(n+1)  R^{\frac{1}{n}-4}}{4 n^2}\rho^4-\frac{(4 n+5) 
	   R^{\frac{1}{n}-3}}{3 n^2}\rho^3+\frac{2 (n+2) 
	   R^{\frac{1}{n}-2}}{n^2}\rho^2-\frac{4 R^{\frac{1}{n}-1}}{n^2} \rho+\frac{(12n^2-11n+17) R^{\frac{1}{n}}}{12 n^2},
\end{equation}
with $\alpha_2=-\sqrt{\frac{8}{\pi}}$ and $\alpha_3=-\sqrt[3]{\frac{3}{\pi}}$ and $R$ sufficiently large; e.g., $R \approx (10^{57}\frac{1}{\textrm{m}})^3N$.
This potential is twice continuously differentiable and globally bounded.

The remaining part of the interaction is called the correlation potential $V_c$. No analytic expression is known for it. However, it is possible to determine the shape of this potential as a function of the density using Quantum Monte Carlo methods; see, e.g., \cite{Attaccalite}. In Figure \ref{QMCforVC}, we plot the shape of $V_c$ used in the numerical experiments in Section \ref{sec:experiments}. 
\begin{figure}[ht]
\begin{center}
	\includegraphics[page=7]{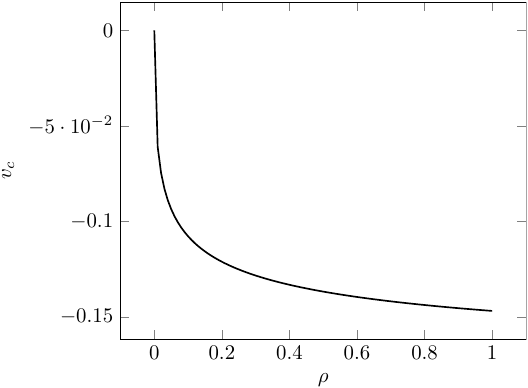}
\end{center}
\caption{The Quantum Monte Carlo fit used in the numerical experiments of this work from \cite{Attaccalite}. The values for the limits are $\lim_{\rho\rightarrow 0} V_c(\rho)=0$ and $\lim_{\rho\rightarrow \infty} V_c(\rho)=-0.1925$.}
\label{QMCforVC}
\end{figure}

All correlation potentials commonly used are of similar structure; see, e.g., \cite{MarquesOliveiraBurnus2012}. They are zero for zero density and otherwise negative, while having a convex shape. Furthermore, they are bounded by $V_x$ in the sense that $|V_c(\rho)|< |V_x(\rho)|$ for all $\rho\in \R^+$. 

It is clear that in applications, confined electron systems subject to external control are of paramount importance. The confinement is obtained considering external potentials such that $\Psi$ is non-zero only on a bounded domain $\Omega$.
For this reason, we denote by 
$V_0$ a confining potential that may represent the attracting potential of the nuclei of a molecule or the walls of a quantum dot. A typical model is the harmonic oscillator potential, 
$V_0(x)=u_0 \, x^2$.
 
A control potential aims at steering the quantum system to change its configuration towards a target state or to optimize the value of a given observable. In most cases, this results in a change of energy that necessarily requires a time-dependent interaction of the electrons with an external electro-magnetic force. For this purpose, we introduce a control potential with the following structure $V_{ext}^c(x,t)=u(t) \, V_u(x)$, where $u(t)$ has the role of a modulating amplitude. A specific case is the dipole control potential, $V_{ext}^c(x,t)=u(t) \, x$. 

In our the TDKS system, we consider the following external potential 
\begin{align*}
	V_{ext}(x,t,u)=u(t) V_u(x)+V_0(x).
\end{align*}
In particular, we consider the control of a quantum dot by a changing gate voltage modeled by a variable quadratic potential, $V_{u}(x)=x^2$, and a laser control in dipole approximation, $V_{u}(x)=x \cdot p$, with a polarization vector $p$.

With this setting, we have completely specified our TDKS 
model. Next, we discuss the corresponding functional 
analytic framework. 

We consider our TDKS model defined on a bounded domain $\Omega \subset \R^n$, $n=2,3$, with $\pa\Omega\in C^2$ and $t\in (0,T)$, together with initial conditions $ \psi_j(\cdot,0)= \psi_j^{0}(\cdot)\in H^1_0(\Omega;\C)$, $j=1,\ldots,N$. For brevity, we denote our KS wavefunction by $\Psi=( \psi_1,\ldots, \psi_N)$, and $\Psi^0=( \psi_1^0,\ldots, \psi_N^0)$. Therefore, we have $|\Psi(x,t)|^2 = \rho(x,t)$.

We define the following function spaces. $X=L^2(0,T; H^1_0(\Omega; \C^N))$ and $W=\{\Psi \in X | \Psi'\in X^*\}$ with the norms $\|\Psi\|_X^2=\int_0^T\|\Psi\|_{L^2}^2+\|\nabla \Psi\|_{L^2}^2 \dd t$ and $\|\Psi\|_W^2=\|\Psi\|_X^2+\|\Psi'\|_{X^*}^2$, where $X^*$ is the dual space of $X$; we also need $Y=L^2(0,T;L^2(\Omega;\C^N))$ which is endowed with the usual norm.

To improve readability of the analysis that follows, we 
write the potentials in \eqref{SumPotentials} as functions of 
$\Psi$ instead of $\rho=|\Psi|^2$. We shall also omit the 
explicit dependence of $V_H$ on $x$ if no confusion may arise.

\begin{lemma}
The exchange potential term $V_x(\Psi)\Psi$ is Lipschitz continuous,
i.e. $\|V_x(\Psi)\Psi-V_x(\Upsilon)\Upsilon\|_{L^2(\Omega;\C^N)}\leq L\|\Psi-\Upsilon\|_{L^2(\Omega;\C^N)}$ and $\|V_x(\Psi)\Psi-V_x(\Upsilon)\Upsilon\|_{X^*}\leq L\|\Psi-\Upsilon\|_X$.
\end{lemma}
\begin{proof}
	The function $f:\C^N\rightarrow \C^N$, $f(z)=V_x(z)z$ is continuously differentiable with bounded derivative, hence Lipschitz continuous with Lipschitz constant $L$ from $\C^N$ to $\C^N$. With this preparation, we have Lipschitz continuity from $L^2$ to $L^2$ 
	as follows 
	\begin{align*}
		\int_\Omega |V_x(\Psi(x,t))\Psi(x,t)-V_x(\Upsilon(x,t))\Upsilon(x,t)|^2 \dd x \leq\int_\Omega L^2|\Psi(x,t)-\Upsilon(x,t)|^2 \dd x .
	\end{align*}
Similarly, we have Lipschitz continuity from $X$ to $X^*$ as follows 
	\begin{align*}
		&\|V_x(\Psi)\Psi-V_x(\Upsilon)\Upsilon\|_{X^*}^2\leq \|V_x(\Psi)\Psi-V_x(\Upsilon)\Upsilon\|_Y^2 \\
		&=\int_0^T \|V_x(\Psi(x,t))\Psi(x,t)-V_x(\Upsilon(x,t))\Upsilon(x,t)\|_{L^2(\Omega;\C^N)}^2\dd t\\
		&\leq\int_0^T L^2\|\Psi(x,t)-\Upsilon(x,t) \|_{L^2(\Omega;\C^N)}^2\dd t = L^2 \|\Psi-\Upsilon\|_Y^2\leq L^2 \|\Psi-\Upsilon\|_X^2,
	\end{align*}
	where we use the Gelfand triple $X\hookrightarrow Y \hookrightarrow X^*$ and the fact that $V_x(\Psi)\Psi\in Y$ as $V_x(\Psi)\in L^\infty(0,T; L^\infty(\Omega;\R))$.
\end{proof}

Throughout the paper, we make the following assumptions on the potentials.
\begin{assumption*}
\begin{enumerate}
	\item \label{assumptionVcBounded} The correlation potential  $V_c$ is uniformly bounded in the sense that\linebreak $|V_c(\Psi(x,t))|\leq K,$ $\forall x\in\Omega$, $t\in [0,T]$, $\Psi\in Y$; this corresponds to the correlation potentials from the Libxc library	\cite{MarquesOliveiraBurnus2012} applying a similar approach to \eqref{Vxcuttoff}.
	\item The correlation potential term $V_c(\Psi)\Psi$ is continuously real-Fréchet differentiable (see Definition \ref{TDKS:defRealDiff}) with bounded derivative, c.f. Lemma \ref{lem:VxPsiFrechet} for the same result on the exchange potential. \label{ass:VcFrechet}
	\item The confining potential and the spacial dependence of the control potential are bounded, i.e. $V_0, V_u \in L^\infty(\Omega;\R)$; as we consider a finite domain, this is equivalent to excluding divergent external potentials.
	\item The control is $u \in H^1(0,T;\R)$. This is a classical assumption in optimal control, see, e.g. \cite{VonWinckelBorzi2008}.
\end{enumerate}	
\end{assumption*} 

The following theorem from \cite{SprengelCiaramellaBorzi2016} states existence and uniqueness of \eqref{KSequations}--\eqref{KSinitcond} in a setting well-suited for optimal control. 
\begin{theorem}\label{thm:uniquesol}
	The weak formulation of \eqref{KSequations}, with $V_{ext}\in H^1(0,T;L^\infty(\Omega;\R))$, $\Psi(0)=\Psi^0 \in L^2(\Omega; \C^N)$,  admits a unique solution in $W$, that is, there 
	exists $\Psi\in X$, $\Psi'\in X^*$, such that
	\begin{equation}\label{eq:KSweak}
	\begin{split}
	i \scalarCN{ \partial_t \Psi(t)}{ \Phi }
	=& \scalarCN{ \nabla \Psi }{ \nabla \Phi }
		+ \scalarCN{ V_{ext}(\cdot,t,u) \Psi }{ \Phi }\\
		& + \scalarCN{ V_{Hxc}(\Psi(t)) \Psi(t)}{ \Phi },
	\end{split}
	\end{equation}
 for all $\Phi \in H_0^1(\Omega; \C^N)$ and a.e. in $(0,T)$.  
	
If $\Psi^0\in H^1_0(\Omega;\C^N)$ and $\partial \Omega \in C^2$, 
then the unique solution to \eqref{eq:KSweak} is as follows 
	\begin{align}
		\Psi \in L^2(0,T; H^2(\Omega; \C^N)) \cap L^\infty(0,T; H^1_0(\Omega; \C^N)) ;
	\end{align}
if in addition $\Psi^0\in H^2(\Omega;\C^N)$, 
 we have $\Psi\in L^\infty(0,T;H^2(\Omega;\C^N)\cap H_0^1(\Omega;\C^N))$.
\end{theorem}

By the continuous embedding $W\hookrightarrow C([0,T];L^2(\Omega; \C^N))$, see e.g. \cite[p. 287]{Evans2010}, the solution is continuous in time. 

Similar problems have been studied in \cite{Jerome2015}:
\begin{theorem}
Assuming that $V_{ext}\in C^1([0,T];C^1(\overline{\Omega};\R))$ and $V_{ext}\geq 0$, and a Lipschitz condition on $V_{xc}$ and a continuity assumption on $V_{xc}$, then \eqref{eq:KSweak} 
with $\Psi^0 \in H_0^1(\Omega; \C^N)$ has a
unique solution in $C([0,T];H^1_0(\Omega; \C^N)) \cap C^1([0,T];H^{-1}(\Omega; \C^N))$.  
\end{theorem}

Taking into account only the Hartree potential but not the exchange-correlation potential, existence of a unique solution in $C([0,\infty); H^2(\R^3;\C))\cap C^1([0,\infty); L^2(\R^3; \C))$ is shown in \cite{CancesLeBris1999}.

As the potential in \eqref{eq:KSweak} is purely real, the norm of the wave function is conserved, see e.g. \cite{SprengelCiaramellaBorzi2016}. We have 
\begin{lemma}\label{lem:normconservation} 
	The $L^2(\Omega;\C)$-norm of the solution to \eqref{eq:KSweak} is conserved in the sense that\linebreak 
$\|\Psi(\cdot,t)\|_{L^2(\Omega;\C^N)}=\|\Psi^0\|_{L^2(\Omega;\C^N)}$ for all $t \in [0,T]$.
\end{lemma}

\section{Formulation of TDKS optimal control problems}\label{sec:oct}

Optimal control of quantum systems is of fundamental importance in quantum mechanics applications. The objectives of the control may be of different nature ranging from the breaking of a chemical bond in a molecule by an optimally shaped laser pulse to the manipulation of electrons in two-dimensional quantum dots by a gate voltage potential. In this framework, the objective of the 
control is modeled by a cost functional to be optimized under 
the differential constraints represented by the quantum model 
(in our case the TDKS equation) including the control mechanism. 

We consider an objective $J$ that includes different target functionals and a control cost as follows 
\begin{equation}\label{oct:functional}\begin{split}
	J(\Psi, u)&=\underbrace{\frac{\beta}{2}\int_0^T \int_\Omega (\rho(x,t)-\rho_d(x,t))^2\dd x \dd t}_{J_\beta} 
	 +\underbrace{\frac{\eta}{2}\int_{\Omega} \chi_A(x) \rho(x,T)\dd x}_{J_\eta} + \underbrace{\frac{\nu}{2}\|u\|_{H^1(0,T;\R)}^2}_{J_\nu}.
\end{split}\end{equation}
The first term, $J_\beta$ models the requirement that the electron density evolves following as close as possible a given target trajectory, $\rho_d$. We remark that $J_\beta$ is only well-defined if $\Psi$ is at least in $L^4(0,T;L^4(\Omega;\C^N))$. This is guaranteed by the improved regularity from Theorem \ref{thm:uniquesol} for $\Psi^0\in H^2(\Omega;\C^N)$. 
The term $J_\eta$ aims at locating the density outside of a certain region $A$. The term $J_\nu$ penalizes the cost of the control. We remark that the regularization term $J_\nu$ can be any weighed $H^1(0,T;\R)$ norm, e.g. $\|u'\|_{L^2(0,T;\R)}^2+a\|u\|_{L^2(0,T;\R)}^2$ for $a> 0$.
We assume that the target weights are all non-negative $\beta,\eta\geq 0$, with $β+\eta> 0$, 
 and the regularization weight $\nu>0$. The characteristic function of $A$ is given by $\chi_A(x)=\begin{cases}1 & x\in A \\ 0 & \text{otherwise}\end{cases}$.

Our purpose is to find an optimal control function $u$, 
which modulates a dipole or a quadratic potential, such that $J(\Psi,u)$ is minimized subject to the constraint that $\Psi$ satisfies the TDKS equations. 
This problem is formulated as follows  
\begin{equation}\label{TDKSopc}
	\min_{(\Psi,u)\in (W,H^1(0,T;\R))}  J(\Psi,u) \qquad \text{ subject to } \eqref{eq:KSweak}.
\end{equation}
Here and in the following, we assume that $\Psi^0 \in H^1(\Omega; \C^N)$.

The solutions to this PDE-constrained optimization problem are characterized as solutions to the corresponding first-order optimality conditions \cite{BorziSchulz2012,TroeltzschEnglish}. 
These conditions for \eqref{TDKSopc} can be formally obtained by setting to zero the gradient of the following Lagrange function
\begin{align}
	&L( \Psi, u, \Lambda)=J( \Psi, u)+ L_1( \Psi, u, \Lambda), 
	\qquad \mbox{ where }   \label{LagrangeL}\\
	&L_1=\Re\left( \sum_{j=1}^N\int_0^T \int_\Omega\left( i\pabl{\psi_j(x,t)}{t}-\left(-\nabla^2+V_{ext}(x,t,u)+V_{Hxc}(x,t,\rho) \right)\psi_j(x,t) \right)\co{\lambda_j(x,t)} \dd x\dd t\right) . \nonumber 
\end{align}
The function $\Lambda=(\lambda_1, \ldots, \lambda_N)$, where $\lambda_j\in L^2(0,T;H^1_0(\Omega;\C))$, $j=1,\ldots,N$, represent the adjoint variables. In the Lagrange formalism, a solution to (\ref{TDKSopc}) corresponds to a stationary point of $L$, where the derivatives of $L$ with respect to $\psi_j$, $\lambda_j$ and $u$ must be zero along any directions $\delta\psi$, $\delta\lambda$, and $\delta u$. A detailed calculation of these derivatives can be found in Appendix \ref{sec:DerivationOfOptSystem}. The main difficulty in the derivation is the complex and non-analytic dependence of the Kohn-Sham potential on the wave function, which results in the terms $\nabla_\psi L_H$, $\nabla_\psi L_{xc}$  in \eqref{optsystem:adjoint} below. 

The first-order optimality conditions define the following optimality system  
\begin{subequations} 
	\begin{align}
		i\pabl{\psi_m(x,t)}{t}&=\left(-\nabla^2+V_{ext}(x,t,u)+V_{Hxc}(x,t,\rho) \right)\psi_m(x,t),   \label{optsystem:forward}\\
		\psi_m(x,0)&=\psi_m^0(x),\label{optsystem:forward:initial}\\
		i \pabl{\lambda_m}{t}&=\left(-\nabla^2+V_{ext}(x,t,u) \right)\lambda_m(x,t)
		+\nabla_\psi L_H+\nabla_\psi L_{xc}
		-2β(\rho-\rho_d)\psi_m, \label{optsystem:adjoint}\\
		\lambda_m(x,T)&=-i 
		η\chi_A(x)  \psi_m(x,T) , \label{optsystem:adjoint:initial} \\
		 & \nu u(t)+\mu(t)=0,		\label{optsystem:optcond}
	\end{align}
\end{subequations} 
where $m=1,\ldots, N$, and 
\begin{align*}
\nabla_\psi L_H:&=\sum_{j=1}^N V_H\left(2\Re(\psi_j(y,t)\co{\lambda_j(y,t)})\right)(x,t) \psi_m(x,t)  +  V_H(\rho)(x,t) \lambda_m(x,t),\\
		\nabla_\psi L_{xc}:&=2\sum_{j=1}^N \pabl{V_{xc}}{\rho}(\rho(x,t)) \psi_m(x,t)  \Re(\psi_j(x,t)\co{\lambda_j(x,t)})    + V_{xc}(\rho(x,t)) \lambda_m(x,t) . 
\end{align*}
Further, $\mu$ is the $H^1$-Riesz representative of the continuous linear functional\linebreak $\scalarT{-\Re\scalarCN{\Lambda}{V_u\Psi}}{\cdot}$; 
see Theorem \ref{thm:gradient} below. 
Assuming that $u\in H^1_0(0,T;\R)$, $\mu$ can be computed by solving the equation
\begin{align*}
	\left(-\abl{^2}{t^2}+1\right) \mu=-\Re\scalarCN{\Lambda}{V_u\Psi}, \quad \mu(0)=0,\, \mu(T)=0,
\end{align*}
which is understood in a weak sense. For more details, see, e.g., \cite{VonWinckelBorzi2008}.

Theorem \ref{thm:uniquesol} guarantees that \eqref{optsystem:forward}--\eqref{optsystem:forward:initial} is uniquely solvable and hence the reduced cost functional 
\begin{align}
	\hat{J}:H^1(0,T;\R) \rightarrow \R, \quad \hat{J}(u):=J(\Psi(u),u) ,
\end{align}
is well defined. 

To ensure the correct regularity properties of the adjoint variables, we do not use the Lagrange formalism explicitly in this work. Instead, we directly use the existence theorem from \cite{SprengelCiaramellaBorzi2016} for the adjoint equation \eqref{optsystem:adjoint}--\eqref{optsystem:adjoint:initial}, an extension of Theorem \ref{thm:uniquesol}, and derive the gradient of the reduced cost functional $\hat{J}$ from these solutions in Theorem \ref{thm:gradient} below. Later, we use this gradient to construct a numerical optimization scheme to minimize $\hat{J}$.

\section{Theoretical analysis of TDKS optimal control problems}\label{sec:theory}
In this section, we present a mathematical analysis of the optimal control problem \eqref{TDKSopc}. To this end, we first show that both the constraint given by the TDKS equations and the cost functional $J$ are continuously real-Fréchet differentiable. Subsequently, we prove existence of solutions to the optimization problem.

The Kohn-Sham potential depends on the density $\rho$. The density is a real-valued function of the complex wavefunction and can therefore not be holomorphic.
As complex differentiability is a stronger property than what we need in the following, we introduce the following weaker notion of real-differentiability. 

\begin{definition}\label{TDKS:defRealDiff}
Let $X,Y$ be complex Banach spaces. A map $f:X\rightarrow Y$ is called real-linear if and only if
	\begin{enumerate}
		\item $f(x)+f(y)=f(x+y)$ $\forall x,y\in X$ and
		\item $f(\alpha x)=\alpha f(x)$ $\forall \alpha \in \R$ and $\forall x\in X$.
	\end{enumerate}
	The space of real-linear maps from $X$ to $Y$ is a Banach space.
	
	We call a map $f:X\rightarrow Y$ real-G\^ateaux (real-Fr\'echet) differentiable if the standard definition of G\^ateaux (Fr\'echet) differentiability holds for a real-linear derivative operator.
\end{definition}

\begin{remark}\label{CH3:TDKS:RealDiffRemark}
	In complex spaces, the notion of real-G\^ateaux (real-Fr\'echet) differentiability is weaker than G\^ateaux (Fr\'echet) differentiability.
	However, all theorems for differentiable functions in $\R^2$ also hold in $\C$ for functions that are just real-differentiable.
	This is the case for all theorems that we will make use of, e.g. the chain rule and the implicit function theorem.
	Therefore, it is enough to show real-Fr\'echet differentiability of the constraint.
	
	An alternative and fully equivalent approach is to consider a real vector
	\begin{align*}
		\hat{\Psi}=\begin{pmatrix}
			\Re \psi_1, \Im \psi_1 , \cdots , \Re \psi_N , \Im \psi_N
		\end{pmatrix}^T
	\end{align*}
	and the corresponding matrix Schrödinger equation.
\end{remark}

\begin{theorem}\label{thm:cFrechet}
The map $c: W\times H^1(0,T;\R) \rightarrow X^*$, defined as
\begin{align}
\label{cectilda}
	c(\Psi, u):=\tilde{c}(\Psi, u)-V_{Hxc}(\Psi)\Psi, \text{ where }
\tilde{c}(\Psi, u):=i\pabl{\Psi}{t}-\left(-\nabla^2+V_0+V_u u(t)\right)\Psi,
\end{align}
is continuously real-Fréchet differentiable.
\end{theorem}
Notice that $c(\Psi, u)=0$ represents the TDKS equation. 

\begin{remark}
We remark that $V_x(\Psi)\Psi$, $V_H(\Psi)\Psi$, and $V_u u\Psi$ are in $Y$. Hence, the operator norm of their derivatives in $\Lop(W,X^*)$ can be bounded in $\Lop(W,Y)$ as follows.
\begin{align*}
	\|B(\Psi)\|_{\Lop(W,X^*)}&=\sup_{ \delta\Psi \in W\setminus\{0\}}\frac{\|B(\Psi) \delta \Psi\|_{X^*}}{\| \delta \Psi\|_W}
	\leq c \sup_{ \delta\Psi \in W\setminus\{0\}} \frac{\|B(\Psi) \delta\Psi\|_Y}{\| \delta\Psi\|_W}\\
	& 
	\leq c\|B(\Psi)\|_{\Lop(W,Y)} \sup_{ \delta\Psi \in W\setminus\{0\}} \frac{\| \delta\Psi\|_Y}{\| \delta\Psi\|_W} 
	\leq  c \, \|B(\Psi)\|_{\Lop(W,Y)} , 
\end{align*}
where $B$ denotes the derivative of $V_x(\Psi)\Psi$, $V_H(\Psi)\Psi$, or $V_u u\Psi$, respectively. 
\end{remark}

Further, to prove Theorem \ref{thm:cFrechet}, we need the following lemmas. We begin with studying the nonlinear exchange potential term.
\begin{lemma}\label{lem:VxPsiGateaux}
	The map $W \ni \Psi \mapsto V_x(\Psi)\Psi \in Y$ is real-Gâteaux differentiable for all $\Psi\in W$ and its real-Gâteaux  derivative, denoted with $A(\Psi)$, is specified as follows 
	\begin{align*}
		&A(\Psi)\in \Lop(W,Y), && A(\Psi)( \delta\Psi)=A_1(\Psi)(\delta\Psi)+A_2(\Psi)(\delta\Psi),\\ &A_1(\Psi)(\delta\Psi):=V_x(\Psi)\delta\Psi,&& A_2(\Psi)(\delta\Psi):=\pabl{V_x}{\rho}2\Re\scalarC{\Psi}{\delta\Psi}\Psi.
	\end{align*}
\end{lemma}
\begin{proof}
First, we need the derivative of the density $\rho$ which is a non-holomorphic function of a complex variable. Differentiation of $\rho$ can be done using the Wirtinger calculus \cite{Remmert1991} where $\Psi$ and $\co{\Psi}$ are treated as independent variables. 
By using these calculus rules, the real-Fréchet derivatives of $\rho=\sum_{m=1}^M  \psi_m \co{ \psi_m}=\scalarC{\Psi}{\Psi}$ are given by
\begin{align*}
	\pabl{\rho}{\Psi}(\delta\Psi) &= 2\Re\scalarC{\Psi}{\delta\Psi},\\
	\pabl{^2\rho}{\Psi^2}(\delta\Psi, \delta \Phi) &= 2\Re\scalarC{\delta\Psi}{\delta\Phi}.
\end{align*}

The directional derivative of $V_x(\Psi)\Psi$ along $\delta \Psi$ is given by $A(\Psi):W\rightarrow X^*$,
\begin{align*}
A(\Psi)( \delta\Psi)&=\lim_{t\rightarrow 0} \frac{1}{t} \left( V_x(\Psi+t \delta \Psi)(\Psi+t \delta \Psi)-V_x(\Psi)\Psi\right)\\
&=\lim_{t\rightarrow 0}\left( V_x(\Psi) \delta \Psi+ \pabl{V_x}{\rho}\pabl{\rho}{\Psi} (\delta\Psi) \Psi+\mathcal{O}(t)\right)\\
&=A_1(\Psi) \delta \psi+A_2(\Psi)( \delta \Psi).
\end{align*}
Using the definition of $V_x$, we have
\begin{align*}
	A(\Psi)( \delta\Psi)=V_x(\Psi) \delta\Psi+\alpha_n\begin{cases}
	\frac{|\Psi|^{2/n-2}}{n}2\Re \scalarC{\Psi}{\delta\Psi}\Psi & |\Psi|^2 \leq R,\\
	\pabl{p}{\rho}(|\Psi|^2)2\Re \scalarC{\Psi}{\delta\Psi}\Psi & R<|\Psi|^2 <2R,\\ 
	0 & |\Psi|^2 \geq 2R
	\end{cases}
\end{align*}
and $\delta\Psi \mapsto A(\Psi)(\delta\Psi)$ is obviously linear in $\delta\Psi$ over the real scalars. 
We are left to show that $A(\Psi)$ is a bounded operator.
\begin{align*}
	\|A(\Psi)\|_{\Lop(W, Y)}& 
	\leq \sup_{ \delta\Psi \in W\setminus\{0\}} \frac{\|A_1(\Psi) \delta \Psi\|_Y}{\| \delta\Psi\|_W}+\sup_{ \delta\Psi \in W\setminus\{0\}} \frac{\|A_2(\Psi)(\delta\Psi)\|_Y}{\| \delta\Psi\|_W}.
\end{align*}
For the first term, we use $V_x(\Psi)\in L^\infty(0,T;L^\infty(\Omega;\R))$ as well as $\|\Psi\|_Y\leq \|\Psi\|_W$ to obtain
\begin{align*}
	\sup_{ \delta\Psi \in W\setminus\{0\}} \frac{\|A_1(\Psi) \delta \Psi\|_Y}{\| \delta\Psi\|_W} \leq \sup_{ \delta\Psi \in W\setminus\{0\}} \frac{\|V_x(\Psi)\|_{L^\infty(0,T;L^\infty(\Omega;\R))}\|\delta\Psi\|_Y}{\|\delta\Psi\|_W}  \leq \|V_x(\Psi)\|_{L^\infty(0,T;L^\infty(\Omega;\R))}.
\end{align*}

For the second term, we decompose the domain into $\Omega=\Omega_1\cup \Omega_2\cup \Omega_3$ depending on the size of $\rho$: $\Omega_1:= \{x\in \Omega |\, |\Psi|^2 \leq R\}$, $\Omega_2:=\{x\in \Omega|\, R < |\Psi|^2 <2R\}$, and $\Omega_3:=\{x\in \Omega|\, |\Psi|^2 \geq 2R\}$.
Using the fact that $|\Psi|^2=\rho$ is bounded by $R$  in $\Omega_1$ and by $2R$ in $\Omega_2$ as well as that $A_2(\Psi)=0$ in $\Omega_3$, and that $\pabl{p}{\rho}$ is monotonically decreasing between $R$ and $2R$ we obtain 
\begin{align*}
	\|A_2(\Psi)(\delta\Psi)\|_Y^2
	&=\int_0^T \int_{\Omega_1} |A_2(\Psi)(\delta\Psi)|^2\dd x\dd t+c\int_0^T \int_{\Omega_2} |A_2(\Psi)(\delta\Psi)|^2\dd x\dd t\\
	&=\int_0^T \int_{\Omega_1} \left|\alpha_n |\Psi|^{2/n-2}2\Re\scalarCN{\Psi}{\delta\Psi}\Psi\right|^2\dd x\dd t\\
	&\quad+\int_0^T \int_{\Omega_2} \left|\alpha_n\pabl{p}{\rho}(\rho)2\Re \scalarC{\Psi}{\delta\Psi}\Psi\right|^2\dd x\dd t\\
	&\leq \int_0^T \int_{\Omega_1} 4\alpha_n^2 |\Psi|^{4/n}|\delta\Psi|^2\dd x\dd t
	+4\alpha_n^2 \left(\pabl{p}{\rho}(R)\right)^2 \int_0^T \int_{\Omega_2} |\Psi|^4 |\delta\Psi|^2 \dd x\dd t\\
	&\leq 4\alpha_n^2 \left(R^{2/n}+ \frac{R^{2/n-2}}{n^2}\right)\|\delta\Psi\|_Y^2.
\end{align*}

We have shown that the directional derivative $\delta\Psi\mapsto A(\Psi)(\delta\Psi)$ is a bounded linear map for all $\Psi\in W$, hence $\Psi \mapsto V_x(\Psi)\Psi$ is real-Gâteaux differentiable.
\end{proof}

Before improving this result to real-Fréchet differentiability, we prove a general result on the $L^2$-norm of products of functions. To this end, denote $Y_1=L^2(0,T;L^2(\Omega;\C))$.
\begin{lemma}\label{lem:l2normproduct}
	Let $k\in \N$ be arbitrary. Given two functions $f\in L^\infty(\Omega;\C)$ and $g\in L^2(\Omega;\C^k)$. Then $\|fg\|_{L^2(\Omega;\C^k)}\leq \frac{\sqrt{\mu(\Omega)}}{\sqrt{2\pi}^n} \|f\|_{L^2(\Omega;\C)}\|g\|_{L^2(\Omega;\C^k)}$. Similarly, for $f\in L^\infty(0,T; L^\infty(\Omega;\C))$ and $g\in Y$, we have$\|fg\|_Y\leq \frac{\sqrt{\mu(\Omega)T}}{\sqrt{2\pi}^{n+1}} \|f\|_{Y_1}\|g\|_Y$.
\end{lemma}
\begin{proof}
The Fourier transform $\mathcal{F}:L^2(\Omega;\C)\rightarrow L^2(\Omega;\C)$ is an isometry by Plancherel's theorem, where we extend the functions by zero outside of $\Omega$. Furthermore, for $f\in L^1(\Omega;\C)$, $g\in L^1(\Omega;\C^k)$ the convolution theorem states $\sqrt{2\pi}^n\F{fg}=\F{f}\star \F{g}$. 
Hence, using the embeddings $L^\infty(\Omega;\C) \hookrightarrow L^2(\Omega;\C) \hookrightarrow L^1(\Omega;\C)$, we find
\begin{align*}
	\|fg\|_{L^2(\Omega;\C)}&=\|\F{fg}\|_{L^2(\Omega;\C)}=\frac{1}{\sqrt{2\pi}^n}\|\F{f}\star \F{g}\|_{L^2(\Omega;\C)}
	\leq \frac{1}{\sqrt{2\pi}^n}\|\F{f}\|_{L^1(\Omega;\C)} \|\F{g}\|_{L^2(\Omega;\C)}\\
	&\leq \frac{\sqrt{\mu(\Omega)}}{\sqrt{2\pi}^n}\|\F{f}\|_{L^2(\Omega;\C)} \|\F{g}\|_{L^2(\Omega;\C)}= \frac{\sqrt{\mu(\Omega)}}{\sqrt{2\pi}^n}\|f\|_{L^2(\Omega;\C)} \|g\|_{L^2(\Omega;\C)},
\end{align*}
	where Young's inequality for convolutions was used, see e.g. \cite[Theorem 14.6]{Schilling2005}.
	The second claim is obtained by applying the same calculations 
	on the domain $\Omega \times (0,T)$.
\end{proof}

\begin{lemma}
\label{lem:VxPsiFrechet}
	The map $W \ni \Psi \mapsto V_x(\Psi)\Psi \in Y$ is continuously real-Fréchet differentiable with derivative $\D(V_x(\Psi)\Psi)=A(\Psi)$, where $A(\Psi)\in \Lop(W, Y)$ is given in Lemma \ref{lem:VxPsiGateaux}.
\end{lemma}
\begin{proof}
We prove that the real-Gâteaux derivative $A(\Psi)$ of $V_x(\Psi)\Psi$ at $\Psi$ is continuous from $W$ to $\Lop(W,Y)$. Then the real-Fréchet differentiability follows immediately from \cite[Proposition A.3]{AndrewsHopper2011}.

Once is proved that $V_x (\Psi) \Psi$ is real-Fréchet differentiable, the real-Gâteaux and the real-Fréchet derivatives coincide, and the continuity of the real-Gâteaux derivative carries over to the real-Fréchet derivative. Hence, we have to show the following 
\begin{align*}
	\forall \epsilon>0, \, \exists \, \delta >0, \text{ such that } \|A(\Psi)-A(\Phi)\|_{\Lop(W,Y)}<\epsilon, \, \forall \, \|\Psi-\Phi\|_{W} <  \delta.
\end{align*}
For this purpose, to ease our discussion, we consider $\tilde{A}_2(\Psi)\psi_j\psi_m$, $j,m=1,\dotsc,N$, where 
\begin{align*}
	\tilde{A}_2(\Psi)=\alpha_n\begin{cases}
	2\frac{|\Psi|^{2/n-2}}{n} & |\Psi|^2 \leq R,\\
	\pabl{p}{\rho}(|\Psi|^2) & R<|\Psi|^2 <2R,\\
	0 & |\Psi|^2 \geq 2R,
	\end{cases}
\end{align*}
such that $\tilde{A}_2(\Psi) \psi_m \sum_{j=1}^N2\Re(\psi_j\co{\delta\psi_j})=\bigl(A_2(\Psi)(\delta\Psi)\bigr)_m$.
For all $\epsilon>0$, $\Psi\mapsto\tilde{A}_2(\Psi)$ is continuously differentiable with bounded derivative for $|\Psi|\geq \epsilon$, hence the same holds for $\tilde{A}_2(\Psi)\psi_j\psi_m$, which is therefore Lipschitz continuous.
In zero, $\tilde{A}_2(\Psi)\psi_j\psi_m$ is Hölder continuous with exponent $\alpha=\frac{2}{n}$,
\begin{align*}
	|\tilde{A}_2(\Psi)\psi_j\psi_m|&=\frac{2\alpha_n}{n}|\Psi|^{2/n-2}|\psi_j\psi_m| \leq \frac{2\alpha_n}{n}|\Psi|^{2/n-2}|\Psi| |\Psi| =\frac{2\alpha_n}{n}|\Psi|^{2/n}.
\end{align*}
Together, we have Hölder continuity for all $\Psi\in W$,
$|\tilde{A}_2(\Psi)\psi_j\psi_m-\tilde{A}_2(\Phi)\phi_j\phi_m|\leq c|\Psi-\Phi|_{\C^N}^\alpha$ and similarly
$|\tilde{A}_2(\Psi)\co{\psi_j}\psi_m-\tilde{A}_2(\Phi)\co{\phi_j}\phi_m|\leq c|\Psi-\Phi|_{\C^N}^\alpha$.

We remark that, if a function $f:\C^N\rightarrow \C^N$ is Hölder continuous with exponent $\alpha$, i.e. $|f(x)-f(y)|<c|x-y|^\alpha$, then $f:L^2(\Omega;\C^N) \rightarrow L^2(\Omega;\C^N)$ is also Hölder continuous with the same exponent as follows
\begin{equation}\label{eq:HoeldercontiscontL2}\begin{split}
	\|f(x)-f(y)\|_{L^2(\Omega;\C^N)}^2&=\int |f(x(t))-f(y(t))|^2\dd t \leq c^2 \int |x(t)-y(t)|^{2\alpha} \dd t\\ &=c^2\|x-y\|_{L^{2\alpha}(\Omega;\C^N)}^{2\alpha}
	 \leq {c'}^2\|x-y\|_{L^2(\Omega;\C^N)} ^{2\alpha}.
\end{split}\end{equation}

As $\tilde{A}_2(\Psi)\psi_j\psi_m\in L^\infty(0,T;L^\infty(\Omega;\C))$ is bounded, we can apply Lemma \ref{lem:l2normproduct}. 
Now, we turn our attention to $A_2(\Psi)$ and use the Hölder continuity of $\tilde{A}_2(\Psi)\psi_j\psi_m$ to obtain the following estimate. 

We have 
\begin{align*}
	&\|A_2(\Psi)(\delta\Psi)-A_2(\Phi)(\delta\Psi)\|_Y^2\\
	&=\sum_{m=1}^N\int_0^T \int_\Omega |\tilde{A}_2(\Psi)\psi_m \sum_{j=1}^N(\psi_j\co{\delta\psi_j}+\co{\psi_j}\delta\psi_j)-\tilde{A}_2(\Phi)\phi_m \sum_{j=1}^N(\phi_j\co{\delta\psi_j}+\co{\phi_j}\delta\psi_j)|^2\dd x\dd t\\
	&\leq\sum_{m=1}^N \int_0^T \int_\Omega \sum_{j=1}^N|\co{\delta\psi_j}( \tilde{A}_2(\Psi)\psi_j\psi_m-\tilde{A}_2(\Phi)\phi_j\phi_m)|^2 \dd x \dd t\\
	&\quad+\sum_{m=1}^N \int_0^T \int_\Omega \sum_{j=1}^N|\delta\psi_j( \tilde{A}_2(\Psi)\co{\psi_j}\psi_m-\tilde{A}_2(\Phi)\co{\phi_j}\phi_m)|^2 \dd x \dd t\\
	&\leq c \sum_{j=1}^N \|\delta\psi_j\|_{Y_1}^2 \sum_{m=1}^N \left(\|\tilde{A}_2(\Psi)\psi_j\psi_m-\tilde{A}_2(\Phi)\phi_j\phi_m)\|_{Y_1}^2+	\|\tilde{A}_2(\Psi)\co{\psi_j}\psi_m-\tilde{A}_2(\Phi)\co{\phi_j}\phi_m\|_{Y_1}^2\right)\\
	&\leq c' \sum_{j=1}^N \|\delta\psi_j\|_{Y_1}^2 \sum_{m=1}^N \|\Psi-\Phi\|_Y^{2\alpha}\\
	&=c' \|\delta \Psi\|_Y^2 N\|\Psi-\Phi\|_Y^{2\alpha}<c' \|\delta \Psi\|_Y^2 N \delta^{2\alpha}.
\end{align*}
Furthermore, by Lemma \ref{lem:l2normproduct} and the Hölder continuity of $V_x$, we have
\begin{align*}
	\|A_1(\Psi)(\delta\Psi)-A_1(\Phi)(\delta\Psi)\|_Y \leq \|A_1(\Psi)-A_1(\Phi)\|_Y \|\delta\Psi\|_Y \leq c''\|\Psi-\Phi\|_Y^{2/n} \|\delta\Psi\|_Y.
\end{align*}
Now, we have the following 
\begin{align*}
	\|A(\Psi)-A(\Phi)\|_{\Lop(W,Y)}&=\sup_{\delta \Psi\in W\setminus\{0\}} \frac{\|A(\Psi)(\delta\Psi)-A(\Phi)(\delta\Psi)\|_Y}{\|\delta \Psi\|_W}\\
	&\leq (\sqrt{c'N}+c'')\delta^\alpha \sup_{\delta \Psi\in W\setminus\{0\}}  \frac{\|\delta \Psi\|_Y}{\|\delta \Psi\|_W} \leq (\sqrt{c'N}+c'')\delta^\alpha =:\epsilon.
\end{align*}
This completes the proof of the real-Fréchet differentiability of $V_x(\Psi)\Psi$.
\end{proof}

\begin{lemma}\label{lem:VHFrechet}
	The map $\Psi \mapsto V_H(\Psi)\Psi$ is continuously real-Fréchet differentiable from $W$ to $Y$ and from $W$ to $X^*$ with derivative $\D(V_H(\Psi)\Psi)(\delta \Psi)=V_H(\Psi) \delta\Psi+\int_\Omega \frac{ 2\Re\scalarC{\Psi}{\delta\Psi}}{|x-y|}\dd y \Psi$.
\end{lemma}
\begin{proof}
	By \cite[Lemma 2]{SprengelCiaramellaBorzi2016}, $V_H(\Psi)\Psi\in Y$.
	The following expansion holds
	\begin{align*}
			V_H(\Psi+ \delta\Psi)&=V_H(\Psi)+\int_\Omega \frac{ 2\Re \scalarC{\Psi}{\delta \Psi}}{|x-y|}\dd y+\int_\Omega \frac{ \scalarC{\delta\Psi}{\delta \Psi}}{|x-y|}\dd y.
	\end{align*}
	Hence, we get
	\begin{align*}
		V_H(\Psi+\delta\Psi)(\Psi+\delta\Psi)&=V_H(\Psi)\Psi+V_H(\Psi)\delta \Psi+\int_\Omega \frac{ 2\Re \scalarC{\Psi}{\delta \Psi}}{|x-y|}\dd y \Psi\\
		&\quad+\int_\Omega \frac{ 2\Re \scalarC{\Psi}{\delta \Psi}}{|x-y|}\dd y \delta \Psi+\int_\Omega \frac{ \scalarC{\delta\Psi}{\delta \Psi}}{|x-y|}\dd y \Psi+ \int_\Omega \frac{ \scalarC{\delta\Psi}{\delta \Psi}}{|x-y|}\dd y\delta\Psi.
	\end{align*}
	
	By proof of Lemma \cite[Lemma 2]{SprengelCiaramellaBorzi2016}, the last three terms are bounded by\linebreak
		$C_N\|\delta\Psi\|_{L^2(\Omega;\C^N)}\|\Psi\|_{H^1(\Omega;\C^N)}\|\delta\Psi\|_{L^2(\Omega;\C^N)}$,
		$C_N\|\delta\Psi\|_{L^2(\Omega;\C^N)}\|\delta\Psi\|_{H^1(\Omega;\C^N)}\|\Psi\|_{L^2(\Omega;\C^N)}$,
		respective	$C_N\|\delta\Psi\|_{H^1(\Omega;\C^N)}\|\delta\Psi\|_{L^2(\Omega;\C^N)}^2$.
	Defining $\D(V_H(\Psi)\Psi)(\delta \Psi):=V_H(\Psi) \delta\Psi+\int_\Omega \frac{ 2\Re\scalarC{\Psi}{\delta\Psi}}{|x-y|}\dd y \Psi$, and using the embedding $W\hookrightarrow C([0,T];L^2(\Omega;\C^N))$, we obtain
	\begin{align*}
		&\frac{\|V_H(\Psi+\delta\Psi)(\Psi+\delta\Psi)-V_H(\Psi)\Psi-\D(V_H(\Psi)\Psi)(\delta \Psi)\|_{X^*}^2}{\|\delta \Psi\|_W^2}\\
		&\leq\frac{\|V_H(\Psi+\delta\Psi)(\Psi+\delta\Psi)-V_H(\Psi)\Psi-\D(V_H(\Psi)\Psi)(\delta \Psi)\|_Y^2}{\|\delta \Psi\|_W^2}\\
		&\leq\frac{1}{\|\delta \Psi\|_W^2}\left(\int_0^T\|\delta\Psi\|_{L^2(\Omega;\C^N)}^2\|\Psi\|_{H^1(\Omega;\C^N)}^2\|\delta\Psi\|_{L^2(\Omega;\C^N)}^2\right.\\
		&\left.+
		\|\delta\Psi\|_{L^2(\Omega;\C^N)}^2\|\delta\Psi\|_{H^1(\Omega;\C^N)}^2\|\Psi\|_{L^2(\Omega;\C^N)}^2+	\|\delta\Psi\|_{H^1(\Omega;\C^N)}^2\|\delta\Psi\|_{L^2(\Omega;\C^N)}^4 \dd t\right)\\
		&\leq \frac{1}{\|\delta \Psi\|_W^2}\left(\max_{0\leq t \leq T}\|\delta\Psi\|_{L^2(\Omega;\C^N)}^4 \left(\|\Psi\|_X^2+\|\delta\Psi\|_X^2\right)\right.\\
		&\left.+\max_{0\leq t \leq T}\|\delta\Psi\|_{L^2(\Omega;\C^N)}^2\max_{0\leq t \leq T}\|\Psi\|_{L^2(\Omega;\C^N)}^2\|\delta\Psi\|_X^2\right)\\
		&\leq \frac{\|\delta\Psi\|_W^4 \left(\|\Psi\|_X^2+\|\delta\Psi\|_X^2\right)+\|\delta\Psi\|_W^2\|\Psi\|_W^2\|\delta\Psi\|_X^2}{\|\delta \Psi\|_W^2}\\
		&\leq  \|\delta\Psi\|_W^2\left(\|\Psi\|_{X}^2+\|\delta\Psi\|_X^2\right)+
		\|\delta\Psi\|_W^2\|\Psi\|_W^2 \rightarrow 0  \text{ for } \|\delta\Psi\|_W \rightarrow 0.
	\end{align*}
Hence, $V_H(\Psi)\Psi$ is real-Fréchet differentiable with derivative $V_H(\Psi) \delta\Psi+\int_\Omega \frac{ 2\Re\scalarC{\Psi}{\delta\Psi}}{|x-y|}\dd y \Psi$ and the derivative is continuous from $W$ to $\Lop(W,Y)$ and from $W$ to $\Lop(W,X^*)$.
\end{proof}

Using these results, we can now prove Theorem \ref{thm:cFrechet} 
that states the real-Fréchet differentiability of the map $c$.

\begin{proof}[\bf Proof of Theorem \ref{thm:cFrechet}]
We prove that $c$ is a real-Fréchet differentiable function of $\Psi$ and $u$. For this purpose, we first consider the map $\tilde{c}$ 
defined in \eqref{cectilda}. Simple algebraic manipulation 
results in the following 
\begin{align*}
	\tilde{c}(\Psi+ \delta \Psi, u+ \delta u)&= \tilde{c}(\Psi, u)+\tilde{c}( \delta \Psi, u)-V_u \delta u  \Psi-V_u  \delta u  \delta \Psi
\end{align*}
Next, we show that the Fréchet derivative of $\tilde{c}$ is given by
\begin{align}\label{ctildefrechet}
	\D\tilde{c}(\Psi, u)( \delta \Psi,  \delta u)=\tilde{c}( \delta \Psi,u)-V_u \delta u \Psi= i\pabl{ \delta \Psi}{t}-\left(-\nabla^2+V_0+V_u u\right) \delta \Psi-V_u \delta u  \Psi.
\end{align}
To bound the reminder $\tilde{c}(\Psi+\delta\Psi,u+\delta u)-\tilde{c}(\Psi,u)- \D\tilde{c}(\Psi,u)(\delta\Psi,\delta u)$, we consider
\begin{align*}
		\|V_u  \delta u  \delta \Psi\|_Y^2 
	\leq \|V_u\|_{L^\infty(\Omega;\R)}^2 \| \delta u\|_{C[0,T]}^2 \| \delta  \Psi\|_X^2,
\end{align*}
where we use the embedding $H^1(0,T;\R)\hookrightarrow C[0,T]$. 
Further, using the estimate $\| \delta u\|_{C[0,T]}\leq c\| \delta u\|_{H^1(0,T;\R)}$ and $\| \delta  \Psi\|_X \le \| \delta  \Psi\|_W $, we can improve the inequality above in the following sense 
\begin{align*}
 \|V_u  \delta u  \delta \Psi\|_Y \leq c' \| \delta u\|_{H^1(0,T;\R)}\| \delta  \Psi\|_W\leq c'\left(\| \delta u\|_{H^1(0,T;\R)}^2+\| \delta  \Psi\|_W^2\right) \leq c'\left(\| \delta u\|_{H^1(0,T;\R)}+\| \delta  \Psi\|_W\right)^2.
\end{align*}
With this we can show the Fréchet differentiability as follows
\begin{align*}
	\frac{\|V_u  \delta u  \delta \Psi\|_Y}{\| \delta u\|_{H^1(0,T;\R)}+\| \delta \Psi\|_W}\leq c'\left(\| \delta u\|_{H^1(0,T;\R)}+\| \delta  \Psi\|_W\right) \rightarrow 0 \text{ for } \| \delta u\|_{H^1(0,T;\R)}+\| \delta  \Psi\|_W \rightarrow 0.
\end{align*}
Hence, $\tilde{c}$ is Fréchet differentiable, and $\D\tilde{c}$ given in \eqref{ctildefrechet} represents its derivative.
The derivative is continuous from $W\times H^1(0,T;\R)$ to $X^*$.

The exchange potential is continuously real-Fréchet differentiable by Lemma \ref{lem:VxPsiFrechet}, the correlation potential by Assumption \ref{ass:VcFrechet}, and the Hartree potential by Lemma \ref{lem:VHFrechet}.
To summarize, we have the following real-Fréchet derivative of $c$. 
\begin{equation}\label{eq:FrechetDc}\begin{split}
	&Dc(\Psi, u)( \delta\Psi,  \delta u)=c( \delta\Psi, u)-V_u \delta u \Psi-\pabl{V_{xc}}{\rho}2\Re \scalarC{\Psi}{\delta\Psi} -\int_\Omega \frac{ 2\Re\scalarC{\Psi}{\delta\Psi}}{|x-y|}\dd y \Psi\\
	&=i\pabl{ \delta\Psi}{t}-\left(-\nabla^2+V_{ext}(x,t,u)+V_{Hxc}(\Psi)\right) \delta\Psi-V_u \delta u \Psi\\
	&\quad-\pabl{V_{xc}}{\rho}2\Re\scalarC{\Psi}{\delta\Psi}\Psi -\int_\Omega \frac{ 2\Re\scalarC{\Psi}{\delta \Psi}}{|x-y|}\dd y \Psi.
\end{split}\end{equation}
\end{proof}

We have discussed the differentiability properties of the 
differential constraint $c$ that are required below to prove the existence of a minimizer and for establishing the gradient of the reduced cost functional. Next, we study the cost functional $J$.

\begin{theorem}\label{lem:JFrechet}
	The cost functional $J:W\times H^1(0,T;\R) \rightarrow \R$ defined in \eqref{oct:functional} is continuously real-Fréchet differentiable.
\end{theorem}
\begin{proof}
	The norm $\|u\|_{H^1(0,T;\R)}^2$ is differentiable by standard results with $\D\left(\frac{\nu}{2}\|u\|_{H^1(0,T;\R)}^2\right)(\delta u)=\nu\scalarTH{u}{\delta u}$.
	The tracking term $J_\beta=\frac{\beta}{2}\int_0^T\int_\Omega (\rho(x,t)-\rho_d(x,t))^2\dd x\dd t$ is a quadratic functional and hence Féchet differentiable with derivative
	\begin{align*}
		D J_\beta (\Psi)(\delta \Psi)=2\beta \|(\rho-\rho_d)\Re\scalarC{\Psi}{\delta\Psi}\|^2_Y.
	\end{align*}
	
	$J_η=\frac{η}{2} \int_\Omega \chi_A(x) \rho(x,T) \dd x$ is well defined, because $\Psi\in C([0,T]; L^2(\Omega;\C^N))$. 
	The directional derivative
	\begin{align*}
		\D J_η(\Psi)(\delta\Psi)=η \int_\Omega\chi_A(x) \Re \scalarC{\Psi(x,T)}{\delta \Psi(x,T)}\dd x
	\end{align*}
	is obviously linear and continuous in $\delta\Psi$, hence it is the real-Gâteaux derivative. Furthermore, we have
	\begin{align*}
		&\frac{|J_η(\Psi+\delta\Psi)-J_η(\Psi)-\D J_η(\Psi)(\delta\Psi)|}{\|\delta\Psi\|_W}=\frac{\left| \int_\Omega \chi_A(x) |\delta \Psi(T)|^2\dd x\right|}{\|\delta\Psi\|_W}\leq\frac{\|\delta \Psi(T)\|_{L^2(\Omega;\C^N)}^2}{\|\delta\Psi\|_W}\\
		&\leq \frac{\max_{0\leq t\leq T}\|\delta \Psi(t)\|_{L^2(\Omega;\C^N)}^2}{\|\delta\Psi\|_W} \leq \frac{c\|\delta \Psi\|_W^2}{\|\delta\Psi\|_W}=c \|\delta\Psi\|_W \rightarrow 0 \text{ for } \|\delta\Psi\|_W \rightarrow 0.
	\end{align*}
Therefore $J_η$ is real-Fréchet differentiable from $W$ to $\R$. The Féchet derivative depends linearly on $\Psi$ and is bounded by
	\begin{align*}
		\|\D J_η(\Psi)\|_{\Lop(W,\R)}&=\eta \sup_{\delta\Psi \in W\setminus\{0\}} \frac{\left|\int_\Omega \chi_A(x)\Re \scalarC{\Psi(x,T)}{\delta \Psi(x,T)}\dd x\right|}{\|\delta\Psi\|_W}\\
		&\leq \eta \frac{\|\Psi(T)\|_{L^2(\Omega;\C^N)}\|\delta\Psi(T)\|_{L^2(\Omega;\C^N)}}{\|\delta\Psi\|_W}
		\leq \eta c^2\frac{\|\Psi\|_W \|\delta \Psi\|_W}{\|\delta \Psi\|_W} =\eta c^2 \|\Psi\|_W.
	\end{align*}
	Hence, the derivative is continuous from $W$ to $\Lop(W,\R)$.
\end{proof}

\subsection{Existence of a minimizer}\label{sec:ExistenceMinimizer}
In this section, we discuss existence of a minimizer of the optimization problem \eqref{TDKSopc}. Uniqueness cannot be expected because, e.g., the phase of the wave function does not appear in the cost functional $J$.

We start by collecting some known facts.
We make use of the following version of the Arzel\`{a}-Ascoli theorem which holds also for general Banach spaces, see, e.g., \cite[Theorem 3.10-2]{Ciarlet2013} and remark thereafter.
\begin{lemma}[Arzel\`{a}-Ascoli] \label{lem:ArzelaAscoli}
	Given a sequence $(f_n)_n$ of functions $f_n\in C(K; \mathcal{B})$ where $\mathcal{B}$ is a Banach space and $K=[0,T]$, which is 
	\begin{enumerate}
		\item uniformly bounded, i.e. $\exists M$, such that $\|f_n\|_{C(K;\mathcal{B})}\leq M$, $\forall \, n\in\N$, and
		\item 	
equicontinuous, i.e. given any $\epsilon>0$, there exists $ \delta(\epsilon)>0$, such that $||f_n(t)-f_n(s)||_\mathcal{B}<\epsilon$ for all $t,s\in K$ with $|t-s|< \delta(\epsilon)$ and all $n\in\N$;
	\end{enumerate}
	then there exists a subsequence $(f_{n_l})_l$ and a function $f\in C(K;\mathcal{B})$ such that
	\begin{align*}
		\lim_{l\rightarrow \infty} \|f_{n_l}-f\|_{C(K;\mathcal{B})} =0.
	\end{align*}
\end{lemma}

For the purpose of our discussion, notice that the semigroup generated by $H_0=-\nabla^2+V_0$ is denoted by $U(t)=e^{-iH_0t}$. For self-adjoint operators $H_0$, as in our case, 
	$U(t)$ is strongly continuous by the Stone's theorem; see  \cite{Stone1932} and \cite[p. 34]{Yserentant2010}. 
	
We need the following lemma; see also \cite{RuggenthalerPenzVanLeeuwen2015}. 
\begin{lemma}\label{lem:Duhamel}
	The Duhamel form of the TDKS equation \eqref{eq:KSweak} is given by
	\begin{align}
	\Psi(t)&=e^{-iH_0t}\Psi(0)-ie^{-iH_0t}\int_0^t g(s)\dd s \quad \text{with } g(s)=
				e^{iH_0s}\left(uV_u\Psi(s)+V_{Hxc}(\Psi(s))\Psi(s)\right).
	\end{align}
\end{lemma}
\begin{proof}
	We follow the approach in \cite{Salomon2005} and write the following
	\begin{align*}
		i\pa_t \Psi(t)&=\left(-\nabla^2 + V_0+u V_u+V_{Hxc}(\Psi(t))\right)\Psi(t)\\
		&=H_0\Psi(t)+uV_u\Psi(t)+V_{Hxc}(\Psi(t))\Psi(t)\\
		\Leftrightarrow		ie^{iH_0t}(\pa_t\Psi(t)+iH_0\Psi(t))&=e^{iH_0t}\left(uV_u\Psi(t)+V_{Hxc}(\Psi(t))\Psi(t)\right)\\
		\Leftrightarrow		\abl{}{t}\left(ie^{iH_0t}\Psi(t)\right)&=e^{iH_0t}\left(uV_u\Psi(t)+V_{Hxc}(\Psi(t))\Psi(t)\right)\\
		\Leftrightarrow		ie^{iH_0t}\Psi(t)-i\Psi(0)&=\int_0^t e^{iH_0s}\left(uV_u\Psi(s)+V_{Hxc}(\Psi(s))\Psi(s)\right)\dd s , 
		\end{align*}
which proves the lemma. 
\end{proof}

Now, we can prove the existence of a solution to (\ref{TDKSopc}).
\begin{theorem}
	The optimal control problem (\ref{TDKSopc}) with $\Psi^0\in H^1_0(\Omega;\C^N)$ and $\beta=0$ admits at least one solution in 
	$(\Psi, u) \in W \times H^1(0,T;\R)$. In the case $\beta \neq 0$, an optimal solution exists in $W \times H^1(0,T;\R)$ assuming that 
	$\Psi^0\in H^2(\Omega;\C^N) \cap H^1_0(\Omega;\C^N)$.
\end{theorem}

\begin{proof}
	For a given control $u\in H^1(0,T;\R)$, we define $\Psi(u)$ as the unique solution to \eqref{eq:KSweak}. Let $(\Psi_n, u_n):=(\Psi(u_n),u_n)$ be a minimizing sequence of $J$, i.e. $\lim_{n\rightarrow \infty}J(\Psi(u_n),u_n)=\inf_{u\in H^1(0,T;\R)} J(\Psi(u),u)$.
	As $J$ is coercive with respect to $u$ in the $H^1$ norm, the sequence $(u_n)_n$ is bounded in $H^1(0,T;\R)$. Hence, we can extract a weakly convergent subsequence again denoted by $(u_n)_n$, $u_n \rightharpoonup \hat{u}$ in $H^1(0,T;\R)$. By the Rellich-Kondrachov theorem, we then have $u_n\rightarrow \hat{u}$ in $C[0,T]$.
	
	As the controls $u_n$ in the sequence above are globally bounded in $H^1(0,T;\R)$,  by  \cite[Theorem 3]{SprengelCiaramellaBorzi2016} we have $\|\Psi_n\|_{X}\leq K$ and $\|\Psi_n'\|_{X^*}\leq K'$, where the constants $K$, $K'$ can be chosen to be independent of $u_n$. 
	Hence, we can extract weakly convergent subsequences, again denoted by $\Psi_n$, $\Psi_n'$, as follows 
	\begin{align*}
		\Psi_n \stackrel{X}{\rightharpoonup} \hat{\Psi}, && \Psi_n' \stackrel{X^*}{\rightharpoonup} \hat{\Psi}'.
	\end{align*}
	By the Rellich-Kondrachov theorem $H^1(\Omega;\C^N)\Subset L^2(\Omega;\C^N)$ and by \cite[1.5.2]{Lions1969}   $W\Subset Y$, hence
	\begin{equation}
		\Psi_n \stackrel{Y}{\rightarrow} \hat{\Psi}.
	\end{equation}
	
	By Lemma \ref{lem:VxPsiFrechet} and \ref{lem:VHFrechet} and Assumption \ref{ass:VcFrechet}, $\Psi\mapsto V_{Hxc}(\Psi)\Psi$ is real-Fr\'echet differentiable, hence continuous from $W$ to $Y$. %$\Lop(W,X^*)$.
	Every $(\Psi_n, u_n)$ solves the Schr\"odinger equation \eqref{eq:KSweak} and with the strong convergence of $u_n$ and $\Psi_n$, we can employ \cite[p. 291]{Ciarlet2013} for the products $u_n\Psi_n$ and $V_{Hxc}(\Psi_n)\Psi_n$.
	By standard results, a sequence converging in the $L^2(0,T)$-norm contains a subsequence that converges a.e.\ in $[0,T]$. Hence, we can extract a subsequence, again denoted by $(\Psi_n)$ such that we can pass to the limit. We have 
	\begin{align*}
		\lim_{n\rightarrow \infty} -\scalarCN[L]{i\pabl{\Psi_n}{t}}{\Phi}+\scalarCN[L]{\nabla\Psi_n}{\nabla\Phi}+\scalarCN[L]{V_0\Psi_n+V_u u_n \Psi_n}{\Phi}+\scalarCN[L]{V_{Hxc}(\Psi_n)\Psi_n}{\Phi}\\
		=-\scalarCN[L]{i\pabl{\hat{\Psi}}{t}}{\Phi}+\scalarCN[L]{\nabla\hat{\Psi}}{\nabla\Phi}+\scalarCN[L]{V_0\hat{\Psi}+V_u \hat{u} \hat{\Psi}}{\Phi}+\scalarCN[L]{V_{Hxc}(\hat{\Psi})\hat{\Psi}}{\Phi}
	\end{align*}
	holds a.e.\ in $[0,T]$ and for all $\Phi\in H^1_0(\Omega; \C^N)$, where $L:=L^2(\Omega;C^N)$.
	Hence, $\Psi(\hat{u})=\hat{\Psi}$ and $(\hat{\Psi}, \hat{u})$ solves \eqref{eq:KSweak}.
	
If $\eta \neq 0$, then the 
	evaluation of $\rho$ at the final time $T$ in the cost $J$ 
	is required, such that strong convergence in $C([0,T]; L^2(\Omega;\C^N))$ is needed.
	So far, we only have weak convergence in $W$, which is continuously embedded into $C([0,T]; L^2(\Omega;\C^N))$, but not compactly embedded. 	
	To overcome this problem, we improve our convergence result by using Lemma \ref{lem:ArzelaAscoli}. 
	The required uniform bound is given by Lemma \ref{lem:normconservation}, we have $\|\Psi_n(t)\|_{L^2}=1$, 
	for all $n$ and all $t\in [0,T]$.
	We are left to show the equicontinuity of $\Psi_n$ for the Arzel\`{a}-Ascoli theorem, i.e. find an $ \delta(\epsilon)$ that does not depend on the $n$ of the sequence. To this end, we take a fixed but arbitrary $\epsilon>0$.
	With $\|\Psi_n\|_{L^2}=1$, the Lipschitz continuity of $V_{xc}$, and estimates in \cite[Lemma 2, Theorem 3]{SprengelCiaramellaBorzi2016}, we have
		\begin{align*}
			\left\|\Psi^0-i\int_0^tg(s)\dd s\right\|_{L^2(\Omega;\C^N)} &\leq \|\Psi^0\|_{L^2(\Omega;\C^N)}+\|g\|_Y
			\leq \|\Psi^0\|_{L^2(\Omega;\C^N)}+\|V_{ext}\Psi_n\|_Y\\
			&+\|V_{Hxc}(\Psi_n)\Psi_n\|_Y\\
			&\leq \|\Psi^0\|_{L^2(\Omega;\C^N)}+(\|u_n\|_{C[0,T]} \|V_u\|_{L^\infty(\Omega;\C^N)}+\|V_0\|_{L^\infty(\Omega;\C^N)})\|\Psi_n\|_Y\\
			&+(L+c\|\Psi_n\|_X^2) \|\Psi_n\|_Y\\
			&\leq C,
		\end{align*}
	where $g$ is defined in Lemma \ref{lem:Duhamel}.
	As $U(t)=e^{-iH_0t}$ is continuous by Stone's theorem, there exists a $ \delta_1$ such that 
	\begin{align*}
		\|U(t)-U(t')\|_{\Lop(L^2, L^2)} < \frac{\epsilon}{2C}  \quad \text{for } |t-t'|< \delta_1.
	\end{align*}
	
	Using the Duhamel form from Lemma \ref{lem:Duhamel}, we find for two different times $t$, $t'$, with $|t-t'|< \delta$, the following
	\begin{align*}
		\|\Psi_n(t)-\Psi_n(t')\|_{L^2(\Omega;\C^N)}&\leq \left\|\left( e^{-iH_0t}-e^{iH_0t'}\right) \left( \Psi^0-i\int_0^tg(s)\dd s \right)\right\|_{L^2(\Omega;\C^N)}\\
		&+\left\|e^{-iH_0t'}\int_{t'}^t g(s)\dd s\right\|_{L^2(\Omega;\C^N)}. 
	\end{align*}	
	The first term is bounded by
	\begin{align*}
		\|U(t)-U(t')\|_{\Lop(L^2(\Omega;\C^N), L^2(\Omega;\C^N))}\left\|\Psi^0-i\int_0^tg(s)\dd s\right\|_{L^2(\Omega;\C^N)} \leq \frac{\epsilon}{2C} C=\frac{\epsilon}{2}.
	\end{align*}

	For the second term, we have
	\begin{align*}
		\left\| \int_{t'}^t g(s) \dd s\right\|_{L^2(\Omega;\C^N)}^2&=\int_\Omega \left|\int_{t'}^t g(s, x)\dd s\right|^2\dd x\leq \int_\Omega\left( \int_{t'}^t |g(s,x)| \dd s\right)^2 \dd x \\
		&=\int_{\Omega} \|g(\cdot, x)\|_{L^1(t',t;\C^N)}^2 \dd x \leq \int_\Omega \sqrt{|t-t'|}^2 \|g(\cdot, x)\|_{L^2(t',t;\C^N)}^2 \dd x\\
	&\leq |t-t'| \int_\Omega \int_{t'}^t |g(s,x)|^2 \dd s \dd x\\
		&=|t-t'| \int_{t'}^t \|e^{iH_0s}\left(u V_u\Psi_n(s)+V_{Hxc}(\Psi_n(s))\Psi_n(s)\right)\|_{L^2(\Omega;\C^N)}^2 \dd s.
			\end{align*}
Using the fact that that $e^{iH_0s}$ is unitary, we obtain 
	\begin{align*}		
		&\left\| \int_{t'}^t g(s) \dd s\right\|_{L^2(\Omega;\C^N)}^2\\
		&\leq  |t-t'| \int_{t'}^t \|\left(V_u\Psi_n(s)+V_{xc}(\Psi_n(s))\Psi_n(s)+V_{H}(\Psi_n(s))\Psi_n(s)\right)\|_{L^2(\Omega;\C^N)}^2 \dd s\\
		&\leq  3|t-t'| \int_{t'}^t C^2 \|u_n\|_{C[0,T]}^2 \|\Psi_n(s)\|_{L^2(\Omega;\C^N)}^2+\|V_{xc}(\Psi_n(s))\Psi_n(s)\|_{L^2(\Omega;\C^N)}^2+\|V_H(\Psi_n(s))\Psi_n(s)\|_{L^2(\Omega;\C^N)}^2 \dd s\\
		&\leq 3|t-t'| \int_{t'}^t C'+L\|\Psi_n(s)\|_{L^2(\Omega;\C^N)}^2+C''\|\Psi_n(s)\|_{H^1(\Omega;\C^N)}^4 \|\Psi_n(s)\|_{L^2(\Omega;\C^N)}^2 \dd s\\
		&\leq K |t-t'|^2,
	\end{align*}
where we used the Lipschitz continuity of $V_{xc}=V_x+V_c$ and \cite[Lemma 2]{SprengelCiaramellaBorzi2016} for the estimate on $V_H$. Furthermore, we used Lemma \ref{lem:normconservation} $\|\Psi_n\|_{L^2(\Omega;\C^N)}=1$, and according to \cite[Theorem 3]{SprengelCiaramellaBorzi2016} we have $\|\Psi(t)\|_{H^1(\Omega;\C^N)}\leq C(\Psi^0)$. Moreover, since the $u_n$ are from a bounded sequence, then $\|u_n\|_{C[0,T]}$ is bounded by a global constant.
		
	Now, we define $ \delta:=\min\{ \delta_1, \frac{\epsilon}{2\sqrt{K}}\}$. Then
	\begin{align}
		\|\Psi_n(t)-\Psi_n(t')\|_{L^2(\Omega;\C^N)} < \frac{\epsilon}{2} + \sqrt{K} \delta \leq \epsilon \quad \forall |t-t'|< \delta \text{ and } \forall \, n\geq 1.
	\end{align}
	This means, that the sequence $(\Psi_n)_n$ is equicontinous. As it is also uniformly bounded, we can invoke the Arzelà-Ascoli theorem (Lemma \ref{lem:ArzelaAscoli}) to conclude that there exists a subsequence $\Psi_{n_l}$ and a function $\hat{\Psi}\in C(0,T;L^2(\Omega;\C))$, such that $\lim_{l\rightarrow \infty }\Psi_{n_l}= \hat{\Psi}$ in $C(0,T;L^2(\Omega;\C))$.
	With the strong convergence in $C(0,T;L^2(\Omega;\C))$, we have
	\begin{align*}
		\rho(\hat{u})(\cdot,T)=\sum_{j=1}^N|\Psi_j(\hat{u})(\cdot, T)|^2=\sum_{j=1}^N|\hat{\Psi}_j(\cdot, T)|^2=\hat{\rho}(\cdot, T).
	\end{align*}
	With this result, we have the convergence of \eqref{eq:exminimizer:liminf} also for a target depending only on the final time.

	For the target $J_\beta$ to be well defined, we need to assume higher regularity. Assuming $\Psi^0\in H^2(\Omega;\C^N)$, we get $\Psi\in L^\infty(0,T;H^2(\Omega;\C^N))$ from \cite{SprengelCiaramellaBorzi2016}. Due to the embedding $L^\infty(0,T;H^2(\Omega;\C^N)) \hookrightarrow L^\infty(0,T;L^\infty(\Omega;\C))$ the wavefunction is therefore globally bounded in space and time by a constant depending on the initial condition and $\|u\|_{H^1(0,T;\R)}$.  As $\|u_n\|_{H^1(0,T;\R)}$ is bounded, we have $\|\Psi_n\|_{L^\infty(0,T;L^\infty(\Omega;\C))}<K$, $\|\hat{\Psi}\|_{L^\infty(0,T;L^\infty(\Omega;\C))}<K$ and the same holds for the squares $\rho_n,\hat{\rho}$.
	
	As $\Psi_n \stackrel{Y}{\rightarrow} \hat{\Psi}$, so does $\co{\Psi}_n \stackrel{Y}{\rightarrow} \co{\hat{\Psi}}$. Hence the product converges in the $L^1$-norm, $\rho_n \stackrel{L^1(0,T;L^1(\Omega;\R))}{\longrightarrow} \hat{\rho}$. Furthermore, $(\rho_n)_n$ converges in the $Y$-norm as follows
	\begin{align*}
		\|\rho_n-\hat{\rho}\|_Y^2&= \|\rho_n(\rho_n-\hat{\rho})+\hat{\rho}(\rho_n-\hat{\rho})\|_{L^1(0,T;L^1(\Omega;\R))}\\
		&\leq \|\rho_n\|_{L^\infty(0,T;L^\infty(\Omega;\R))}\|\rho_n-\hat{\rho}\|_{L^1(0,T;L^1(\Omega;\R))}+\|\hat{\rho}\|_{L^\infty(0,T;L^\infty(\Omega;\R))}\|\rho_n-\hat{\rho}\|_{L^1(0,T;L^1(\Omega;\R))}.
	\end{align*}
	As $J_\beta(\Psi)=\frac{\beta}{2}\|\rho-\rho_d\|_Y^2$ and the norm is continuous, we can pass the limit.
	
	The norm $\|u\|_{H^1(0,T;\R)}$ is weakly lower semicontinuous. Hence, $(\hat{\Psi}, \hat{u})$ minimizes \eqref{TDKSopc},
	\begin{equation*}\label{eq:exminimizer:liminf}
		J(\hat{\Psi},\hat{u}) \leq \liminf_{n\rightarrow \infty} J(\Psi_n, u_n) = \inf_{u\in H^1(0,T;\R)} J(\Psi(u), u). \qedhere
	\end{equation*}
\end{proof}

\subsection{Necessary optimality conditions}

In this section, we discuss the Lagrange multiplier $\Lambda$ and  state the first-order optimality condition for a minimum.

We start showing the existence of a Lagrange multiplier in the Lagrange framework.
\begin{theorem}
	Given a control $u$ and a corresponding state $\Psi$, there exists a Lagrange multiplier $\Lambda \in X$ associated with $(\Psi,u)$.
\end{theorem}
\begin{proof}
	The constraint $c$ and the objective $J$ are continuously real-Fréchet differentiable by Lemma \ref{lem:JFrechet} and Theorem \ref{thm:cFrechet}, and the derivative $\D c(\Psi, u): W\times H^1(0,T;\R)\rightarrow X^*$ is surjective by \cite{SprengelCiaramellaBorzi2016}, where the results can readily extended to a nonzero right-hand side $F\in X^*$.
	Hence the constraint qualification of Zowe and Kurcyusz is fulfilled \cite{ZoweKurcyusz1979}. Therefore, we have the existence of a  Lagrange multiplier $\Lambda \in X$; see, e.g., \cite[Section 6.1]{TroeltzschEnglish}.
\end{proof}
As we need higher regularity, namely $\Lambda \in W$, we do not make further use of this result. Instead, we proceed in a different way using the existence of a unique solution of the adjoint equation in $W$ from \cite{SprengelCiaramellaBorzi2016}.

As $c(\Psi,u)=0$ is uniquely solvable, we have the following equivalent formulation of the optimization problem.
\begin{lemma} The minimization problem \eqref{TDKSopc} is equivalent to the unconstrained minimization of the reduced cost functional $\hat{J}(u):=J(\Psi(u),u)$,
	\begin{align}
		\min_{(\Psi, u)\in W\times H^1(0,T;\R)} J(\Psi, u), \text{ s.t. } c(\Psi, u)=0
		\quad \Leftrightarrow \quad  \min_{u\in H^1(0,T;\R)} \hat{J}(u).
	\end{align}
\end{lemma}

To calculate the gradient of the reduced cost functional, we make use of the implicit function theorem; see, e.g., \cite[p. 548]{Ciarlet2013}. We apply the chain rule for the real-Fréchet derivative to the cost functional as follows. 
\begin{lemma}\label{lem:formofgradient}
	The derivative of the reduced cost functional is given by
	\begin{align}
		\D \hat{J}(u) (\delta u)&=\D_\Psi J(\Psi(u), u) \D_u\Psi(u) (\delta u)+ \D_u J(\Psi(u),u) (\delta u),
	\end{align}
	where $\D_u\Psi(u) \delta u= \delta\psi$ is the solution to the linearized equation
	\begin{align}
		\bigl(\D c(\Psi,u)\bigr)( \delta\Psi,  \delta u)=\D_\Psi c(\Psi, u) (\delta \Psi)+\D_u c(\Psi, u) (\delta u)=0.
	\end{align}
\end{lemma}
\begin{proof}
The chain rule for the real-Fréchet derivative is given by
\begin{align*}
	\D \hat{J}(u)=\D_\Psi J(\Psi(u), u) \D_u\Psi(u)+ \D_u J(\Psi(u),u).
\end{align*}

To show that the term $\D_u \Psi(u)$ is well defined, we use the implicit function theorem, which ensures differentiability of the map $u\mapsto \Psi(u)$, $H^1(0,T;\R)\rightarrow W$. To this end, we show that the Fréchet derivative $\D_\Psi c(\Psi, u):W\rightarrow X^*$ is a bijection at any $(\Psi,u)\in W\times H^1(0,T;\R)$.

	 We consider the derivative of $c$ with respect to $ \delta\Psi$. This is given by 
	\begin{align*}
	\D_\Psi c(\Psi,u)(\delta\Psi)&=i\pabl{ \delta\Psi}{t}-\left(-\nabla^2+V_{ext}(x,t,u)+V_{Hxc}(\Psi)\right) \delta\Psi
	-\pabl{V_{xc}}{\rho}2\Re\scalarC{\Psi}{\delta\Psi}\Psi\\
	&\quad -\int_\Omega \frac{ 2\Re\scalarC{\Psi}{\delta \Psi}}{|x-y|}\dd y \Psi.
	\end{align*}
	As shown in \cite{SprengelCiaramellaBorzi2016}, the equation $\D_\Psi c(\Psi,u)(\delta\Psi)=F$ with the initial condition $\delta \Psi(t=0)=0$ is uniquely solvable for any right-hand side $F\in X^*$. Hence $\D_\Psi c(\hat{\Psi},\hat{u})$ is a bijection.
	
Now, from $c(\Psi,u)=0$, we have 
\begin{align*}
	\bigl(\D c(\Psi,u)\bigr)( \delta\Psi,  \delta u)=\D_\Psi c(\Psi, u) (\delta \Psi)+\D_u c(\Psi, u) (\delta u)=0.
\end{align*}
Solving this equation for $\delta\Psi$ results in the following 
\begin{align}
\label{dPsix}
	 \delta \Psi=-\bigl(\D_\Psi c(\Psi,u) \bigr)^{-1} \bigl(\D_u c(\Psi, u) \bigr)  (\delta u) = \D_u \Psi(u)  (\delta u).
\end{align}
This means that the solution $ \delta \Psi$ of the linearized equation $\D c=0$ is in fact $\D_u\Psi(u)$.

All together, we find
\begin{align*}
	\D \hat{J}(u) (\delta u)&=\D_\Psi J(\Psi(u), u) \D_u\Psi(u) (\delta u)+ \D_u J(\Psi(u),u) (\delta u)\\
	&=\D_\Psi J(\Psi(u),u) (\delta \Psi)+\D_u J(\Psi(u),u) (\delta u),
\end{align*}
where $ \delta\Psi$ is given by \eqref{dPsix}.
\end{proof}

\begin{theorem}\label{thm:gradient}
The $H^1$-gradient of the reduced optimization problem is given by
	\begin{align}\label{thm:gradient:gradformula}
		\nabla \hat{J}(t) = \nu u(t)+\mu(t),
	\end{align}
where $\mu$ is the $H^1$-Riesz representative of the continuous linear functional\linebreak $\scalarT{-\Re\scalarCN{\Lambda}{V_u\Psi}}{\cdot}$, 
$\Psi$ is the unique solution of $c(\Psi,u)=0$, and $\Lambda$ is the unique solution of $a(\Psi,u,\Lambda)=0$, where
	\begin{equation}\label{eq:adjoint}\begin{split}
		a(\Psi,u,\Lambda):&=i\pabl{\Lambda}{t}-\left(-\nabla^2 +V_{ext}(x,t,u)+V_{Hxc}(\Psi)\right)\Lambda-\pabl{V_{xc}}{\rho}2\Re\scalarC{\Psi}{\Lambda}\Psi\\
		&\quad-\int_\Omega \frac{2\Re\scalarC{\Psi}{\Lambda}}{|y-x|}\dd y\Psi+2\beta(\rho-\rho_d)\Psi ,
	\end{split}\end{equation}
	with the terminal condition
	\begin{align}\label{eq:adjoint:terminal}
		i\Lambda(T)= \eta \, \chi_A \,\Psi(T). 
	\end{align}
\end{theorem}

\begin{proof} We want to calculate the Riesz representative of $\D \hat{J}(u) (\delta u)$. By Lemma \ref{lem:formofgradient}, we have
\begin{align*}
	\D \hat{J}(u) (\delta u)&=\D_\Psi J(\Psi(u), u) \D_u\Psi(u) (\delta u)+ \D_u J(\Psi(u),u) (\delta u).
\end{align*}
The derivative $\D_u J(\Psi(u),u) (\delta u)=\nu \scalarTH{u}{\delta u}$, hence the gradient of the second term is given by
\begin{align}\label{thm:gradient:regpart}
	\scalarTH{\nabla_u J(\Psi, u)}{\delta u}=\nu \scalarTH{u}{\delta u}.
\end{align}
To express the first term as an operator acting on $\delta u$, we need to use the fact from Lemma \ref{lem:formofgradient}, that $\D_u\Psi(u) (\delta u)=\delta \Psi$ is the solution of the linearized equation.
Using the directional derivative calulated in the Appendix \ref{appendix:targetfunctional}, we obtain for the first term
\begin{align}
	&\D_\Psi J(\Psi,u)(\delta \Psi)\nonumber\\
	&=\beta \int_0^T\int_\Omega  (\rho-\rho_d)2\Re\scalarC{\Psi}{\delta \Psi} \dd x\dd t+\int_\Omega 
	 \eta\chi_A \Re\scalarC{\Psi(T)}{\delta\Psi(T)}\dd x\nonumber\\
	&=\beta 2\Re\scalar[Y]{(\rho-\rho_d)\Psi}{\delta\Psi}
	+ 
	\Re\scalarCN{ 
	\eta\chi_A \Psi(T)}{\delta\Psi(T)}.\label{thm:gradient:eq1}
\end{align}
To simplify the equation, we focus on the second term in \eqref{thm:gradient:eq1}. By the continuous embedding $W\hookrightarrow C([0,T]; L^2(\Omega;\C^N))$, we can invoke the fundamental theorem of calculus in time.
\begin{align*}
	\Lambda(T)\co{\delta\Psi(T)}
	=\Lambda(0)\co{\delta\Psi(0)}+\int_0^T \abl{}{t}\left(\Lambda(t)\co{\delta\Psi(t)}\right)\dd t
	=\Lambda(0)\co{\delta\Psi(0)}+
	 \int_0^T\pabl{\Lambda(t)}{t}\co{\delta\Psi(t)}+\Lambda(t)\co{\pabl{\delta\Psi(t)}{t}}\dd t.
\end{align*}
Observing that $\delta\Psi(0)=0$ and using equations \eqref{eq:adjoint}, \eqref{eq:adjoint:terminal} for $\Lambda$ and the fact \eqref{eq:FrechetDc} that $\delta\Psi$ is the solution of the linearized equation, as well as the previous result, we obtain
\begin{align*}
	&\int_\Omega 
	\eta \chi_A\Re\scalarC{\Psi(T)}{\delta\Psi(T)}\dd x 
	= 
	\Re\scalarCN{ 
	\eta \chi_A \Psi(T)}{\delta\Psi(T)} 
	=\Re\scalarCN{i \Lambda(T)}{\delta\Psi(T)}\\
	&=\Re\int_0^T\scalarCN{(-\nabla^2+V_{ext}(u)+V_{Hxc}(\Psi))\Lambda+\pabl{V_{xc}}{\rho}2\Re\scalarC{\Psi}{\Lambda}\Psi+\int_\Omega \frac{2\Re\scalarC{\Psi}{\Lambda}}{|y-x|}\dd y\Psi\right.\\
	&\left. \phantom{\int_\Omega}-2\beta(\rho-\rho_d)\Psi}{\delta\Psi}\dd t
	\\
	&-\Re\int_0^T\left(\Lambda(t),\left(-\nabla^2+V_{ext}(x,t,u)+V_{Hxc}(\Psi)\right) \delta\Psi+V_u \delta u \Psi
		+\pabl{V_{xc}}{\rho}2\Re\scalarC{\Psi}{\delta\Psi}\Psi \right.\\
	&\quad +\left.\int_\Omega \frac{ 2\Re\scalarC{\Psi}{\delta \Psi}}{|x-y|}\dd y \Psi\right)_{L^2} \dd t.
\end{align*}
Using integration by parts with the zero boundary condition, this simplifies to
\begin{align*}
	\int_\Omega 
	\eta \chi_A \Re\scalarC{\Psi(T)}{\delta\Psi(T)}\dd x
	=2\Re\scalar[Y]{-\beta(\rho-\rho_d)\Psi}{\delta\Psi}-\Re\scalar[Y]{\Lambda}{V_u\delta u\Psi}.
\end{align*}
Using this result in \eqref{thm:gradient:eq1}, we obtain
\begin{align}
	\D_\Psi J(\Psi,u)(\delta \Psi)
	=-\Re\scalar[Y]{\Lambda}{V_u\delta u\Psi}
	=\scalarTH{\mu}{\delta u} \label{thm:gradient:statepart}
\end{align}
as, by definition of $\mu$, $\scalarT{-\Re\scalarCN{\Lambda}{V_u\Psi}}{\delta u}=\scalarTH{\mu}{\delta u}$.

Adding \eqref{thm:gradient:statepart} and \eqref{thm:gradient:regpart} together, we obtain \eqref{thm:gradient:gradformula}.
\end{proof}

\begin{theorem}
	Given a local solution $(\Psi, u)\in W\times H^1(0,T;\R)$ of the minimization problem \eqref{TDKSopc}, i.e. $\hat{J}(u)\leq \hat{J}(\tilde{u})$ for all $\|u-\tilde{u}\|_{H^1(0,T;\R)}<\epsilon$ for some fixed $\epsilon>0$. Then there exists a unique Lagrange multiplier $\Lambda\in W$, such that the following first order optimality system is fulfilled.
	\begin{subequations}\begin{align}
		c(\Psi,u)=0, && \Psi(0)&=\Psi^0, \label{thm:necess:c}\\
		a(\Psi,u,\Lambda)=0, && i\Lambda(T)&= 
		\eta \chi_A \Psi(T),\label{thm:necess:a}\\
		\nabla \hat{J}(u)=0,\label{thm:necess:grad}
	\end{align}\end{subequations}
	where $\nabla \hat{J}(u)$ is given by Theorem \ref{thm:gradient}.
\end{theorem}

\begin{proof}
For all admissible pairs $(\Psi, u)$ that satisfy \eqref{thm:necess:c}, the adjoint problem \eqref{thm:necess:a} has a unique solution in $W$ by \cite{SprengelCiaramellaBorzi2016}. As $\hat{J}$ is differentiable, a local minimum is characterized 
by a zero gradient; see, e.g., \cite{BorziSchulz2012}, hence \eqref{thm:necess:grad} holds.
\end{proof}

\section{Numerical approximation and optimization schemes}
\label{sec:implementation}
In this section, we illustrate an approximation scheme for the 
TDKS equation and its adjoint and discuss an optimization algorithm to solve the optimality system \eqref{optsystem:forward}--\eqref{optsystem:optcond} (or equivalently \eqref{thm:necess:c}--\eqref{thm:necess:grad}).

To solve the TDKS equations \eqref{optsystem:forward}, we use the Strang splitting \cite{StrangFaou2014} given by
\begin{equation}\label{Numeric:Strang}\begin{split}
	\psi_j'&=e^{i \delta t  \nabla^2}e^{-i \frac{ \delta t}{2} V(\Psi(t),t)}\psi(t)_j,\\
	\psi_j(t+ \delta t)&=e^{-i\frac{ \delta t}{2} V(\Psi',t+ \delta t)}\psi_j',
\end{split}\end{equation}
where $V=V_{ext}+V_{Hxc}$ and $j=1,\dotsc,N$.
To solve the adjoint TDKS equation \eqref{optsystem:adjoint}, we have to include the inhomogeneous right-hand side as follows.
\begin{equation}\label{Numeric:StrangAdjoint}
\begin{split}
	\psi_j'&=e^{-i\frac{\delta t}{2}\nabla^2}
	\left( e^{-i\frac{\delta  t}{2}\nabla^2}e^{i \frac{\delta t}{2}V(\Psi(t),t)} \psi_j(t)+i \delta t g_j(t- \delta t/2) \right),\\
	\psi_j(t- \delta t)&
	 =e^{i \frac{\delta t}{2}V(\Psi',t- \delta t)}\psi_j',
\end{split}
\end{equation}
with a right-hand side $g(t)$, e.g. $g_j(t)=-2\beta (\rho(t)-\rho_d)\psi_j(t)$. The Laplacian $-\nabla^2$ is evaluated spectrally. 
With this setting, we obtain a discretization scheme that 
provides second-order convergence in time and analytic convergence in space. This is proved for constant potentials in \cite{BaoJinMarkowich2002}, and there is numerical evidence that this accuracy performance also holds in the case of variable potentials.

Furthermore, this time-splitting scheme is unconditional stable and norm preserving, as well as time reversible and gauge invariant. The latter means that adding a constant to the potential changes only the phase of the wave function in such a way that discrete quadratic observables are not changed \cite{BaoJinMarkowich2002}.

We are not able to give a proof for the convergence rate 
of the time-splitting discretization scheme in the presence 
of nonlinearity as it appears in the TDKS equation. However, we refer to results on the Strang splitting for similar SE. For the inhomogeneous case, the importance of an inhomogeneity that vanishes at the boundary is stressed in \cite{StrangFaou2014}. In \cite{Thalhammer2012} the  Strang splitting for the Gross-Pitaevskii equation is studied, and second-order convergence in time and spectral convergence in space is proved. Further, 
quantum models with Lipschitz nonlinearities are studied in \cite{Besse2002}, where second-order convergence in time is obtained. 

Even though these results cannot be applied directly to our problem, 
they suggest that similar accuracy can be expected in our case. Therefore, we study this accuracy issue numerically. To this end, we solve the TDKS equations for two interacting particles in a harmonic trap $V_{ext}=50x^2$. The used initial condition is given by the coherent states of  two non-interacting particles in the harmonic oscillator.
We set $\Omega=(0,L)^2$, $L=7$, and the time interval is $[0, 0.1]$.
Since an analytic solution is not available, we consider a reference solution $\Psi_{\text{reference}}$ obtained solving the problem on a very fine mesh.
In Figure \ref{fig:convergencePDE}, we report results of numerical experiments showing second-order  convergence in time (slope factor $1.995$) and spectral convergence in space.

\begin{figure}[ht]
	\begin{subfigure}[b]{0.47\textwidth}
		\includegraphics[page=1,width=\textwidth]{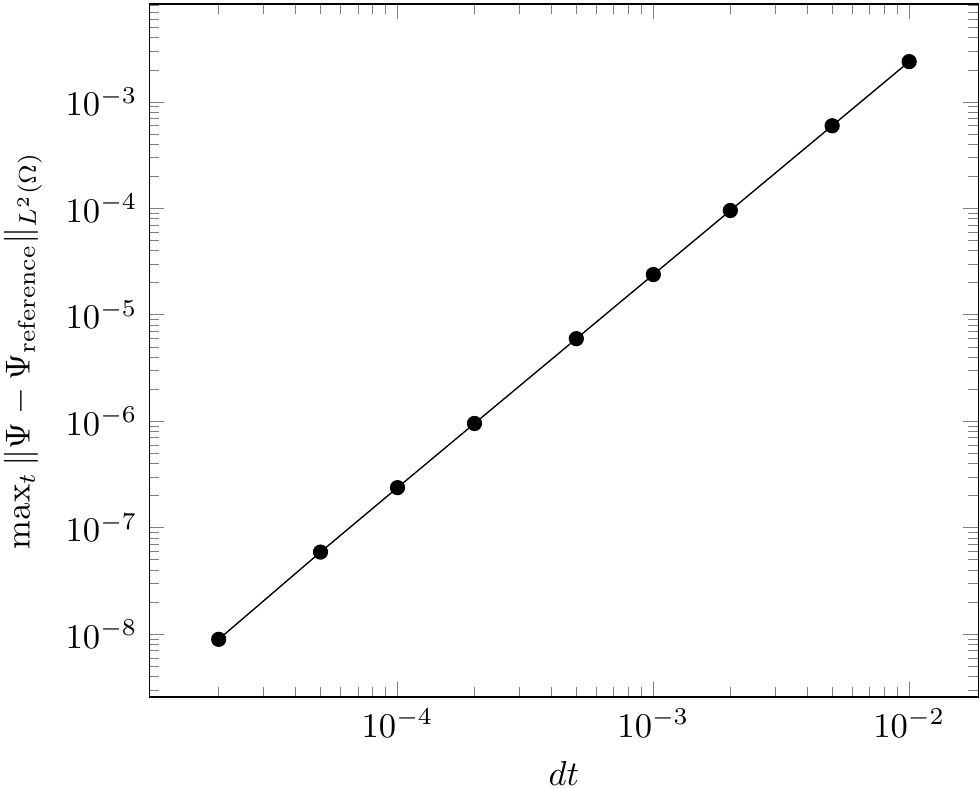}
		\caption{Solution error for the TDKS equation depending 
		on the time step size; $h=0.1186$.}
	\end{subfigure}
	\begin{subfigure}[b]{0.47\textwidth}
		\includegraphics[page=2,width=\textwidth]{Plots}
	\caption{Solution error for the TDKS equation depending 
		on the space mesh size; $dt=5\cdot10^{-6}$.}
	\end{subfigure}
	\caption{Accuracy of the numerical TDKS solution.}
	\label{fig:convergencePDE}
\end{figure}

Next, we illustrate our implementation of the nonlinear conjugate gradient (NCG) method to solve our optimization problem. 
We follow the approach in \cite{HagerZhang}. The minimization algorithm is given in Algorithm \ref{alg:NCG}, where 
Algorithm \ref{alg:Gradient} is called to compute the 
reduced gradient.

\begin{algorithm}[H] 
\caption{(TDDFT optimization with NCG)}
\label{alg:NCG}
\begin{algorithmic}
\small
  \REQUIRE Admissible initial control $u^0(t)$;
  \ENSURE Optimal control $u^{opt}(t)$;
	\STATE Set $n=0$;
  \WHILE{$n < n_{max}$}
  \STATE
  \begin{itemize}\itemsep0em
  	\item Set $n \leftarrow n+1$;
	  \item Calculate the gradient of the reduced cost functional $\nabla_u \hat{J}(u^n)$ by Algorithm \ref{alg:Gradient};
	  \item If norm of gradient $\|\nabla_u \hat{J}(u^n)\|_{H^1}$ is smaller than tolerance, break;
	  \item Use the Hager-Zhang scheme \cite{HagerZhang} to find a new decent direction $d^n$;
	  \item Find a step length $\alpha^n$ by a  line search (we use the method from \cite[p. 60-61]{NocedalWright2006}) along this direction. 
	  	\item Update the control $u^{n}=u^{n-1}+\alpha^n d^n$;
  \end{itemize}
  \ENDWHILE
\end{algorithmic}
\end{algorithm}

\begin{algorithm}[H] 
\caption{(Gradient of reduced cost functional)}
\label{alg:Gradient}
\begin{algorithmic}
\small
  \REQUIRE control $u(t)$, $\Psi^0$, $\Lambda^T$;
  \ENSURE $\nabla \hat{J}(t)$;
  \STATE Solve the forward equation \eqref{optsystem:forward} for the given control $u(t)$ to obtain $\Psi(x,t)$;
  \STATE Solve the adjoint equation \eqref{optsystem:adjoint} for the given control $u(t)$ and the solution of the forward equation $\Psi(x,t)$ to obtain $\Lambda(x,t)$;
  \STATE The gradient is given by \eqref{thm:gradient:gradformula} with the given $u(t), \Psi(x,t), \Lambda(x,t)$;	
\end{algorithmic}
\end{algorithm}

\section{Numerical experiments}\label{sec:experiments}
In this section, we present results of numerical experiments to 
numerically validate our optimization framework.

In all experiments, we consider $N=2$ interacting electrons in $n=2$, so $x=(\xone, \xtwo)^T\in \R^2$. 
Confining electrons to a two-dimensional surface region is used to model quantum dots, which are nowadays widely used in applications; see, e.g., \cite{HansonKouwenhovenPettaTaruchaVandersypen2007} for a review on quantum dots.

As initial guess for the control, we take $u_0= 0$. Different 
results with different choices of the optimization parameters 
are denoted differently. In the figures below, the 
results obtained with $\nu=10^{-5}$ are shown with dotted lines; 
in the case $\nu=10^{-6}$ we plot dash-dotted lines, and with $\nu=10^{-7}$, we use dashed lines. 

In our first experiment, our objective is that the density of 2 electrons follows a prescribed trajectory ($\beta=1$, $\eta=0$). 
Our target trajectory is produced by an oscillating strength of the harmonic confinement $V_{ext}(x,t,u)=50 (\xone^2+ \xtwo^2)+u_d(t)(\xone^2+ \xtwo^2)$ with a forcing $u_d$. Therefore, 
our purpose is to track the density resulting from the prescribed forcing term $u_d$ (solid line in Figure \ref{num:tracking:control}). The stopping criterion for convergence is $\|\nabla J\|_{H^1(0,T;\R)}<5\cdot 10^{-7}$.

The results of this experiment are presented in Figure \ref{num:tracking}. We see that the trajectory is tracked more closely for smaller values of $\nu$, which require 
larger computational effort, while with larger $\nu$ the stopping criterion is met after fewer iterations. 

\begin{figure}[ht]
	\begin{subfigure}[b]{0.55\textwidth}
		\includegraphics[page=4,width=\textwidth]{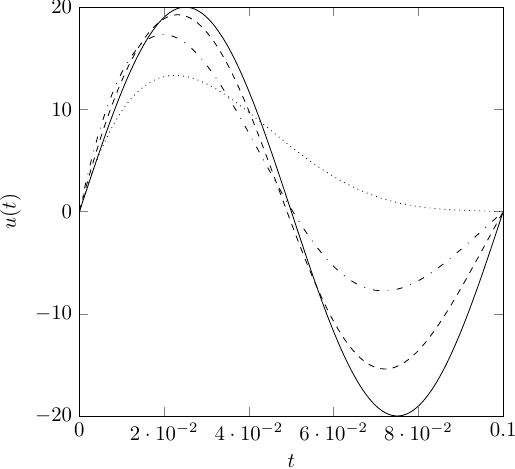}
		\caption{The control $u(t)$ for different weights $\nu$ as well as the control of the target density $\rho_d$.}
		\label{num:tracking:control}
	\end{subfigure}
	\begin{subfigure}[b]{0.4\textwidth}
		\includegraphics[page=3,width=\textwidth]{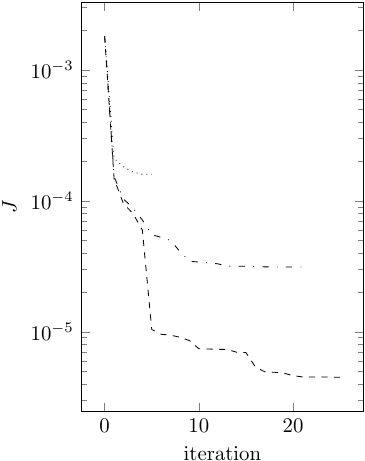}
		\caption{The cost functional $J$ is considerably reduced depending on $\nu$. 
		}
		\label{num:tracking:cost}
	\end{subfigure}
	\caption{The tracking problem.}\label{num:tracking}
\end{figure}

The second experiment is as in \cite{CastroWerschnikGross2012}. 
In this case, we consider two electrons in the 
following asymmetric double well potential 
\begin{align*}
	V_0(x)=\frac{1}{64}\xone^4-\frac{1}{4}\xone^2+\frac{1}{32}\xone^3+\frac{1}{2}\xtwo^2.
\end{align*}
At $t=0$, the electrons are in their ground state which is centred around the global minimum at $(x_1,x_2)=(-3.6,0)$. 
Our objective is to spatially shift this ground state, that is, 
to move the 2 electrons to the right-half space, $\xone>0$. 
For this purpose, we consider the following cost functional 
\begin{align*}
	J=\frac{\eta}{2} \int_{\xone<0} \rho(x,T) \dd x+\frac{\nu}{2} \|u\|_{H^1(0,T;\R)}^2.
\end{align*}
The stopping criterion for convergence is $\|\nabla J\|_{H^1(0,T;\R)}<5\cdot 10^{-5}$.

As the results presented in Figure \ref{num:castro:plots} show, the cost functional can be reduced by approximately 4 orders of magnitude and the density is almost completely localized in the desired set, since $\int_{\xone<0} \rho(x,T) \dd x=2.6\cdot 10^{-4}$. 
This result is obtained with $\nu=10^{-7}$. For smaller values of $\nu$, the stopping criterion is matched later and the objective can be 
further improved.

\begin{figure}[ht]
	\begin{subfigure}[b]{0.55\textwidth}
		\includegraphics[page=6,width=\textwidth]{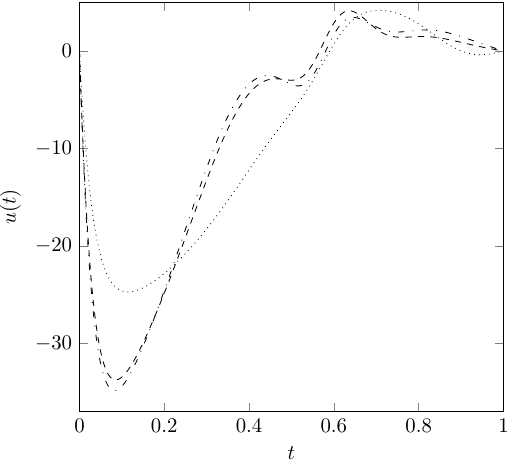}
		\caption{The amplitude of the electric control field $u(t)$ for different weight parameters $\nu$.}
	\end{subfigure}
	\begin{subfigure}[b]{0.4\textwidth}
		\includegraphics[page=5,width=\textwidth]{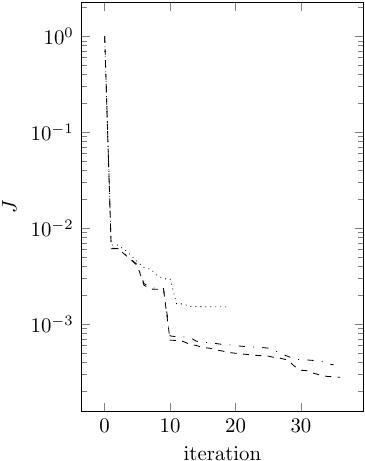}
		\caption{The cost functional as a function of the iteration for different $\nu$.}
	\end{subfigure}
	\caption{Moving the density to a certain region in space with a laser pulse.}\label{num:castro:plots}
\end{figure}

\section{Conclusion}\label{sec:conclusion}
In this paper, an optimal control framework for the 
time-dependent Kohn-Sham model was presented and analyzed. 
The purpose of this work was to provide a mathematical rigorous proof of the existence of optimal controls and their 
characterization as solutions to the corresponding 
optimality system. 
For this purpose, the differentiability properties of the nonlinear Kohn-Sham potential and of the TDKS equation were investigated. 
A proof of the existence of a minimizer was given and 
the first-order optimality system was discussed. 
To validate the proposed optimization framework, an 
efficient Strang splitting discretization scheme was 
implemented and validated. The optimization problem 
was solved using a NCG scheme. 
Results of numerical experiments demonstrated the 
control ability of our scheme with different control problems.

\appendix
\section{Derivation of the optimality system}
\label{sec:DerivationOfOptSystem}
The derivation of the optimality system is done by calculating the directional derivatives of the Lagrange functional 
\eqref{LagrangeL}  
with respect to the states $\psi_j$, the adjoint variables $\lambda_j$, and the control $u$.

\subsection{The forward equation (derivative by \texorpdfstring{$\lambda_j$}{lambda})}
Since $J$ does not depend on $\Lambda$ and $L_1$ depends linearly on $\lambda_j$, so the derivative by $\lambda_j$ gives the TDKS equation for $\psi_j$:
\begin{align}
	i\pabl{\psi_j(x,t)}{t}=\left(-\nabla^2+V_{ext}(x,t,u)+V_{Hxc}(x,t,\rho) \right)\psi_j(x,t).
\end{align}

\subsection{The adjoint equation (derivative by \texorpdfstring{$\psi_j$}{phi})}\label{subsec:DerivationOfAdjoint}
Both $L_1$ and the target term $J$ depend on $\Psi$. We start with $L_1$ which is split into two terms: the linear part $L_2$ and the nonlinear part from the Kohn-Sham potential $L_3$. We have 
\begin{align*}
	L_1&=\underbrace{\Re\left( \sum_{j=1}^N\int_0^T \int_\Omega\left( i\pabl{\psi_j(x,t)}{t}-\left(-\nabla^2+V_{ext}(x,t,u)\right)\psi_j(x,t) \right)\co{\lambda_j(x,t)} \dd x\dd t\right)}_{L_2}\\ &\quad-\underbrace{\Re\left( \sum_{j=1}^N\int_0^T \int_\Omega V_{Hxc}(x,t,\rho) \psi_j(x,t) \co{\lambda_j(x,t)} \dd x\dd t \right)}_{L_3}.
\end{align*}

We begin with the linear part $L_2$.
Differentiating $L_2$ by $\psi_j$ gives the TDKS equation for $\lambda_j$ as it is linear in $\psi_j$ and the sign from integration by parts is canceled by the complex conjugation. One boundary term $B_1$ appears.
\begin{align}
	\nabla_{\psi_j} L_2=i\pabl{\lambda_j(x,t)}{t}-\left(-\nabla^2+V_{ext}(x,t,u) \right)\lambda_j(x,t)+B_1.
\end{align}

For the calculation details, we fix one particle index $j$:
\begin{align*}
	&\skalar{\nabla_{\psi_j} L_2}{\delta \psi_j}=\lim_{α\rightarrow 0^+}\frac{L_2(\{\psi_k+α \delta_{kj}\delta\psi_j\},u)-L_2(\{\psi_k\},u)}{α}\\
	&=\lim_{α\rightarrow 0^+}\frac{1}{α}\Re \sum_{k}\int_0^T \int_\Omega\left(i \pabl{(\psi_k+\delta_{kj}α \delta \psi_k)}{t}-\left( -\nabla^2+V_{ext}\right)(\psi_k+\delta_{kj}α\delta \psi_k)\right)\co{\lambda_k} \dd x\dd t\\
	&-\Re \sum_{k}\int_0^T \int_\Omega\left(i \pabl{\psi_k}{t}-\left( -\nabla^2+V_{ext}\right)\psi_k\right)\co{\lambda_k} \dd x\dd t\\
	&=\Re\int_0^T \int_\Omega \left( i \pabl{\delta \psi_j}{t}-\left(-\nabla^2+V_{ext} \right)\delta\psi_j \right)\co{\lambda_j} \dd x \dd t.
\end{align*}

Now, we use integration by parts and use the fact that 
$\lambda_j$ and $\psi_j$ are zero on the boundary. We obtain 
\begin{align*}
	\skalar{\nabla_{\psi_j} L_2}{\delta \psi_j}&=\Re\int_0^T \int_\Omega \left( i \pabl{\lambda_j}{t}-\left(-\nabla^2+V_{ext} \right)\lambda_j \right)\co{\delta\psi_j} \dd x \dd t\\
	&+ \Re \int_\Omega -i \bigl(\lambda_j(x,T)\co{\delta\psi_j(x,T)}-\lambda_j(x,0)\co{\delta\psi_j(x,0)}\bigr)\dd x.
\end{align*}
At time $t=0$, all wave functions have to fulfill the initial condition, hence $\delta\psi_j(x,0)=0$. We are left with one boundary term 
\begin{align}
	B_1=-\Re \int_\Omega i \lambda_j(x,T)\co{\delta\psi_j(x,T)}\dd x,
\end{align}
which can be removed by prescribing the terminal condition $\lambda_j(x,T)=0$ in the case of $\eta=0$, 
 otherwise by the terminal condition derived below. We obtain 
\begin{align}
	\skalar{\nabla_{\psi_j} L_2}{\delta \psi_j}=\Re\int_0^T \int_\Omega \left( i \pabl{\lambda_j}{t}-\left(-\nabla^2+V_{ext} \right)\lambda_j \right)\co{\delta\psi_j} \dd x \dd t+B_1.
\end{align}

The derivative of the nonlinear part $L_3$ is given in the Lemmas \ref{lem:VxPsiGateaux} and \ref{lem:VHFrechet}.

\subsubsection{Derivative of the target functional}\label{appendix:targetfunctional}
Consider the trajectory term $J_β=\frac{β}{2}\int_0^T\int_\Omega(\rho(x,t)-\rho_d(x,t))^2\dd x\dd t$, and notice the 
following 
\begin{align*}
	|\psi_j+α \delta \psi_j|^2&=|\psi_j|^2+2\Re (\co{\psi_j}α \delta \psi_j)+|α \delta \psi_j|^2,\\
	|\psi_j+α \delta \psi_j|^4&=|\psi_j|^4+4|\psi_j|^2\Re (\co{\psi_j}α \delta \psi_j) +\mathcal{O}(α^2).
\end{align*}
With this preparation, we habe
\begin{align*}
	&\skalar{\nabla_{\psi_j} J_β}{\delta \psi_j}
	=\lim_{α\rightarrow 0+}\frac{β}{2α}\int_0^T\int_\Omega \left(|\psi_j+α \delta \psi_j|^2+(\sum_{i\neq j}|\psi_i|^2-\rho_d) \right)^2 -\frac{β}{2}(\rho-\rho_d)^2\dd x\dd t\\
	&=\lim_{α\rightarrow 0+}\frac{β}{2α}\int_0^T\int_\Omega \left( |\psi_j+α \delta \psi_j|^4+(\sum_{i\neq j}|\psi_i|^2-\rho_d)^2+2|\psi_j+α \delta \psi_j|^2(\sum_{i\neq j}|\psi_i|^2-\rho_d)\right) -\frac{β}{2}(\rho-\rho_d)^2\dd x\dd t\\
	&=2β \int_0^T\int_\Omega|\psi_j|^2\Re (\co{\psi_j} \delta \psi_j) + \Re (\co{\psi_j} \delta \psi_j)(\sum_{i\neq j}|\psi_i|^2-\rho_d)\dd x\dd t\\
	&=2β \Re \int_0^T\int_\Omega(\rho-\rho_d)\psi_j \co{\delta\psi_j} \dd x\dd t \tag{b2}.
\end{align*}

Similarly, we find for the terminal term $J_\eta=\int_\Omega \rho(x,T)\chi_A(x)\dd x$ the following 
\begin{align*}
	\skalar{\nabla_{\psi_j} J_\eta}{\delta \psi_j}
	&=\lim_{α\rightarrow 0+}\frac{η}{2α}\int_\Omega \chi_A\left(\sum_{i=1}^N|\psi_i(x,T)+α\delta_{ij} \delta \psi_j(x,T)|^2-\sum_{i=1}^N|\psi_i(x,T)|^2 \right)\dd x\\
	&=η \Re \int_\Omega \chi_A(x) \psi_j(x,T) \co{\delta\psi_j(x,T)} \dd x.
\end{align*}

\bibliographystyle{amsalphaarxiv}
\bibliography{Literatur}
\end{document}